\documentclass{amsart}

\usepackage{amssymb, amscd, amsthm, mathrsfs}

\newtheorem{MainThm}{Theorem}

\newtheorem{Thm}{Theorem}[section]
\newtheorem{Prop}[Thm]{Proposition}

\newtheorem{Def}[Thm]{Definition}
\theoremstyle{remark}
\newtheorem{Sit}[Thm]{Situation}
\newtheorem{Rmk}[Thm]{Remark}
\newtheorem*{Notation}{Notation}
\newtheorem*{Ack}{Acknowledgement}

\numberwithin{equation}{section}

\newcommand{\Order}{\mathcal{O}}
\newcommand{\into}{\hookrightarrow}
\newcommand{\onto}{\twoheadrightarrow}
\newcommand{\isomto}{\overset{\sim}{\to}}
\newcommand{\isomfrom}{\overset{\sim}{\leftarrow}}
\newcommand{\compose}{\mathrel{\circ}}
\newcommand{\tensor}{\otimes}

\newcommand{\closure}[1]{\overline{#1}}

\newcommand{\Z}{\mathbb{Z}}
\newcommand{\Q}{\mathbb{Q}}

\newcommand{\sep}{\mathrm{sep}}

\newcommand{\et}{\mathrm{et}}

\newcommand{\fppf}{\mathrm{fppf}}

\newcommand{\rat}{\mathrm{rat}}

\newcommand{\PDual}{\mathrm{PD}}

\newcommand{\SDual}{\mathrm{SD}}
\newcommand{\LDual}{\mathrm{LD}}
\newcommand{\sAb}{\mathrm{sAb}}
\newcommand{\Gm}{\mathbf{G}_{m}}

\newcommand{\EtGp}{\mathrm{EtGp}}

\newcommand{\EtGpf}{\mathrm{EtGp}^{f}}

\newcommand{\ind}{\mathrm{ind}}
\newcommand{\pro}{\mathrm{pro}}
\newcommand{\tor}{\mathrm{tor}}

\newcommand{\sm}{\mathrm{sm}}

\newcommand{\set}{\mathrm{set}}

\newcommand{\alg}[1]{\mathbf{#1}}
\newcommand{\dirlim}{\varinjlim}
\newcommand{\invlim}{\varprojlim}

\newcommand{\var}{\;\cdot\;}

\newcommand{\SmGp}{\mathrm{SmGp}}
\mathchardef\mhyphen="2D

\DeclareMathOperator{\Gal}{Gal}
\DeclareMathOperator{\Hom}{Hom}

\DeclareMathOperator{\Ker}{Ker}
\DeclareMathOperator{\Coker}{Coker}
\DeclareMathOperator{\Ext}{Ext}
\DeclareMathOperator{\Spec}{Spec}
\DeclareMathOperator{\Ab}{Ab}
\DeclareMathOperator{\Set}{Set}
\DeclareMathOperator{\Res}{Res}

\DeclareMathOperator{\sheafhom}{\alg{Hom}}
\DeclareMathOperator{\sheafext}{\alg{Ext}}
\DeclareMathOperator{\Pic}{Pic}

\DeclareMathOperator{\Ch}{Ch}

\newcommand{\BetweenThmAndList}{\leavevmode}

\title[N\'eron models of 1-motives and duality]
{N\'eron models of 1-motives and duality}
\author{Takashi Suzuki}
\thanks{
	The author is a Research Fellow of Japan Society for the Promotion of Science
	and supported by JSPS KAKENHI Grant Number JP18J00415.
}
\address{
	Department of Mathematics, Chuo University,
	1-13-27 Kasuga, Bunkyo-ku, Tokyo 112-8551, JAPAN
}
\email{tsuzuki@gug.math.chuo-u.ac.jp}
\date{January 26, 2019}

\begin{document}

\begin{abstract}
	In this paper, we propose a definition of N\'eron models of arbitrary Deligne 1-motives over Dedekind schemes,
	extending N\'eron models of semi-abelian varieties.
	The key property of our N\'eron models is that
	they satisfy a generalization of Grothendieck's duality conjecture in SGA 7
	when the residue fields of the base scheme at closed points are perfect.
	The assumption on the residue fields is unnecessary
	for the class of 1-motives with semistable reduction everywhere.
	In general, this duality holds after inverting the residual characteristics.
	The definition of N\'eron models involves careful treatment of ramification of lattice parts
	and its interaction with semi-abelian parts.
	This work is a complement to
	Grothendieck's philosophy on N\'eron models of motives of arbitrary weights.
\end{abstract}

\maketitle

\tableofcontents


\section{Introduction}
\label{sec: Introduction}

\subsection{Aim of the paper}
\label{sec: Aim of the paper}

Let $X$ be an irreducible Dedekind scheme with function field $K$.
Let $U$ be either a dense open subscheme of $X$ or equal to $\Spec K$.
Recall from \cite[(10.1.10)]{Del74} that
a smooth $1$-motive $M$ over $U$ in the sense of Deligne
is a complex of group schemes $[Y \to G]$ over $U$
whose degree $-1$ term $Y$ is a lattice (\'etale locally isomorphic to $\Z^{n}$ for some $n$)
and degree $0$ term $G$ is an extension of an abelian scheme by a torus.
Raynaud \cite{Ray94} studied monodromy
(i.e.\ the defect of good reduction around $X \setminus U$) of $1$-motives.

In this paper, we define a certain model $\mathcal{N}(M)$ of $M$ over $X$,
which we call the \emph{N\'eron model} of $M$,
generalizing N\'eron (lft) models of semi-abelian varieties \cite[Chap.\ 10]{BLR90}
(see also \cite{Hol16}, \cite{Ore18} for more recent studies).
Grothendieck, in \cite[Exp.\ IX, \S 0.1]{Gro72}, imagined a possibility of a theory of N\'eron models
of motives of arbitrary weights.
On the other hand, there have been several studies of N\'eron models of Hodge structures such as \cite{GGK10}.
Our study of N\'eron models of $1$-motives is a complement to such studies.
We hope that our study sheds some light on possible N\'eron models of more general motives.

One key property of our N\'eron model $\mathcal{N}(M)$ is that
it satisfies a generalization of \emph{Grothendieck's duality conjecture} \cite[IX, Conj.\ 1.3]{Gro72}
when the residue fields of $X \setminus U$ are perfect.
This conjecture is originally for $M = A$ an abelian variety,
in which case (with perfect residue fields)
it is solved by the author \cite{Suz14} after many partial results by other researchers.
By Bertapelle-Bosch \cite{BB00}, the conjecture in its original form (for abelian varieties)
may fail when a residue field is imperfect.
Without the assumption on residue fields,
the original conjecture is true if $A$ has semistable reduction everywhere by Werner \cite{Wer97}
or after inverting the residual characteristics by Bertapelle \cite{Ber01}.
We prove that our N\'eron models $\mathcal{N}(M)$ of $1$-motives $M$ satisfy a duality
under the same assumptions as those results
(i.e.\ for the case of semistable $M$ and for the case of residual characteristics being inverted).
The duality results we prove strongly suggest that our definition of N\'eron models is ``correct''.
If N\'eron models of more general motives make any sense,
then it will be a very interesting problem to try to generalize Grothendieck's duality conjecture to such models.

Our N\'eron model $\mathcal{N}(M)$ represents, in the derived category of $X_{\sm}$,
the truncation $\tau_{\le 0} R j_{\ast} M$
in degrees $\le 0$ of the derived pushforward of $M$
by the natural morphism $j \colon U_{\sm} \to X_{\sm}$.
Here $X_{\sm}$ is the smooth site of $X$,
i.e.\ the category of smooth $X$-schemes with $X$-scheme morphisms
endowed with the \'etale topology,
and $U_{\sm}$ similarly.
Hence $\mathcal{N}(M)$ encodes $j_{\ast} Y$, $R^{1} j_{\ast} Y$ and
the kernel of the morphism $R^{1} j_{\ast} Y \to R^{1} j_{\ast} G$.
The sheaf $R^{1} j_{\ast} Y$ has finite stalks
and contains information about (possibly wild) ramification of the lattice $Y$.
If $Y$ is unramified along $X \setminus U$, then $R^{1} j_{\ast} Y = 0$,
and $\mathcal{N}(M)$ is simplified as $[j_{\ast} Y \to j_{\ast} G]$.

Bosch-Xarles \cite[Def.\ 4.1]{BX96}
defines the N\'eron model of a complex of sheaves $C$ on (the local rigid-analytic version of) $U_{\sm}$
as $R^{0} j_{\ast} C$.
Including information about the degree $-1$ term (or $j_{\ast} Y$) is a new feature of the present work.
Our duality contains the results of Xarles \cite{Xar93} and Bertapelle-Gonzal\'ez-Avil\'es \cite[Thm.\ 1.1]{BGA15}
as a special case where $M$ is a torus.
The result of Xarles mentioned here is essentially about $\tau_{\le 1} R j_{\ast} Y$.
Hence the information of the whole $\tau_{\le 0} R j_{\ast} M$ is crucial
in order to even formulate duality.

According to Gonz\'alez-Avil\'es, Xarles made an (unsuccessful) attempt in 1996
to generalize his result \cite{Xar93} to arbitrary $1$-motives.
The present work has been done independently of his attempt.


\subsection{Main results}
\label{sec: Main results}

Now we state our results.
Let $j \colon U_{\sm} \to X_{\sm}$ and $K$ as above.
Denote the category of $1$-motives over $U$ by $\mathcal{M}_{U}$,
which has a natural additive functor to
the bounded derived category $D^{b}(U_{\sm})$ of sheaves on the site $U_{\sm}$.
Let $\SmGp / X$ be the category of commutative separated smooth group schemes over $X$.
It has a natural additive functor
to the bounded derived category $D^{b}(X_{\sm})$ of sheaves on the site $X_{\sm}$
and hence inherits the notion of quasi-isomorphism of complexes from $D^{b}(X_{\sm})$.
Denote the resulting localization of the category of bounded complexes in $\SmGp / X$ by $D^{b}(\SmGp / X)$.
See Def.\ \ref{def: localize category of groups} for a more detailed definition
and why $D^{b}(\SmGp / X)$ is triangulated.
We have a natural triangulated functor $D^{b}(\SmGp / X) \to D^{b}(X_{\sm})$.
The existence of N\'eron models of semi-abelian varieties
(i.e.\ representability of the sheaf $j_{\ast} G$)
is generalized to $1$-motives as follows.

\begin{MainThm} \label{main: definition of Neron models}
	There exists a canonical additive functor
	$\mathcal{N} \colon \mathcal{M}_{U} \to D^{b}(\SmGp / X)$
	such that the diagram
		\[
			\begin{CD}
					\mathcal{M}_{U}
				@>> \mathcal{N} >
					D^{b}(\SmGp / X)
				\\
				@VVV
				@VVV
				\\
					D^{b}(U_{\sm})
				@> \tau_{\le 0} R j_{\ast} >>
					D^{b}(X_{\sm})
			\end{CD}
		\]
	is commutative.
\end{MainThm}

This means that the complex of sheaves $\tau_{\le 0} R j_{\ast} M$ is
represented by a complex of separated smooth group schemes over $X$,
which is unique up to quasi-isomorphism
and behaves functorially in $M$ in the derived category.
The construction of $\mathcal{N}(M)$ for $M = [Y \to G] \in \mathcal{M}_{U}$ needs, as auxiliary data,
a finite \'etale covering $V$ of $U$
such that $Y \times_{U} V$ extends to a lattice over the normalization of $X$ in $V$
(which means that $V$ kills ramification of $Y$ along $X \setminus U$).
To each such choice of $V$,
we assign a certain canonical complex $\mathcal{N}(M, V)$ in $\SmGp / X$ with terms in degrees $-1$ and $0$
representing $\tau_{\le 0} R j_{\ast} M$.
As an object of $D^{b}(\SmGp / X)$, this complex does not depend on $V$.

Actually this canonical complex $\mathcal{N}(M, V)$ is more useful
than the object $\mathcal{N}(M)$ of $D^{b}(\SmGp / X)$ that it represents,
since functoriality in triangulated categories is difficult to use for some purposes.
For example, the mapping cone of the morphism $\mathcal{N}_{0}(M) \to \mathcal{N}(M)$
mentioned below will be constructed using this actual complex representative.
Nonetheless, the well-definedness of $\mathcal{N}(M)$ makes sense only in $D^{b}(\SmGp / X)$.

The representability of the terms of $\mathcal{N}(M, V)$ is important;
otherwise we would not have much control of the fiber of $\mathcal{N}(M)$
(and $\mathcal{P}(M)$ mentioned below) over $Z$
(see Prop.\ \ref{prop: premorphisms preserving finite products} \eqref{item: derived pull of representables}
and Prop.\ \ref{prop: derived pull of smooth groups}).
Just having a complex of sheaves representing $\tau_{\le 0} R j_{\ast} M$ is not sufficient in this regard.

Next, to state our duality results,
assume that $U \subset X$ is open
(so either $U \ne \Spec K$ or $X$ has finitely many points)
with reduced complement $i \colon Z \into X$.
For $M = [Y \to G] \in \mathcal{M}_{U}$,
let $\mathcal{Y}_{0}$ be the extension by zero of $Y$ along $j \colon U \into X$
and $\mathcal{G}_{0}$ the maximal open subgroup scheme of the N\'eron model of $G$ along $j$
with connected fibers.
(We do not use the more standard notation $\mathcal{G}^{0}$,
in order to avoid confusion with the zeroth term of a complex,
in this highly derived categorical paper.)
Define the \emph{connected N\'eron model} of $M$ by
$\mathcal{N}_{0}(M) = [\mathcal{Y}_{0} \to \mathcal{G}_{0}] \in \SmGp / X$.
We will define a canonical morphism $\mathcal{N}_{0}(M) \to \mathcal{N}(M)$ in $D^{b}(\SmGp / X)$.
There is a canonical mapping cone of this morphism.
This cone is supported on $Z$ (up to quasi-isomorphism).
The fiber over $Z$ of this cone is a complex of \'etale group schemes in degrees $-1$ and $0$
with finitely generated groups of geometric points.
Denote this complex of \'etale group schemes over $Z$ by $\mathcal{P}(M) \in D^{b}(Z_{\et})$
and call it the \emph{N\'eron component complex} of $M$.
We have a canonical distinguished triangle
	\[
			\mathcal{N}_{0}(M)
		\to
			\mathcal{N}(M)
		\to
			i_{\ast} \mathcal{P}(M)
	\]
in $D^{b}(X_{\sm})$.
Let $M^{\vee} \in \mathcal{M}_{U}$ be the dual $1$-motive of $M$
(\cite[(10.2.12), (10.2.13)]{Del74}).
Denote the derived tensor product by $\tensor^{L}$, shift of complexes by $[1]$
and the derived sheaf-Hom functor by $R \sheafhom$.
We will define canonical morphisms
	\begin{gather*}
				\mathcal{N}_{0}(M^{\vee}) \tensor^{L} \mathcal{N}(M)
			\to
				\Gm[1],
		\\
				\mathcal{P}(M^{\vee}) \tensor^{L} \mathcal{P}(M)
			\to
				\Z[1]
	\end{gather*}
in $D(X_{\sm})$, $D(Z_{\et})$, respectively.
They induce morphisms
	\begin{gather*}
				\zeta_{M}
			\colon
				\mathcal{N}(M^{\vee})
			\to
				\tau_{\le 0}
				R \sheafhom_{X_{\sm}}(\mathcal{N}_{0}(M), \Gm[1]),
		\\
				\zeta_{0 M}
			\colon
				\mathcal{N}_{0}(M^{\vee})
			\to
				\tau_{\le 0}
				R \sheafhom_{X_{\sm}}(\mathcal{N}(M), \Gm[1]),
		\\
				\eta_{M}
			\colon
				\mathcal{P}(M^{\vee})
			\to
				R \sheafhom_{Z_{\sm}}(\mathcal{P}(M), \Z[1]).
	\end{gather*}
If the residue field of $Z$ at a point $x \in Z$ has characteristic $p \ge 0$,
then by the residual characteristic exponent of $Z$ at $x$,
we mean $p$ if $p > 0$ and $1$ if $p = 0$.

\begin{MainThm} \label{main: duality} \BetweenThmAndList
	\begin{enumerate}
		\item \label{main: item: trivial duality}
			$\zeta_{M}$ is an isomorphism.
		\item \label{main: item: equivalence}
			$\zeta_{0 M}$ and $\zeta_{0 M^{\vee}}$ are both isomorphisms
			if and only if $\eta_{M}$ is an isomorphism
			if and only if $\eta_{M^{\vee}}$ is an isomorphism.
		\item \label{main: item: semistable case}
			$\eta_{M}$ is an isomorphism if $M$ is semistable
			(meaning that $Y$ is unramified and $G$ is semistable along $j$).
		\item \label{main: item: after inverting res char}
			$\eta_{M} \tensor \Z[1 / n]$ is an isomorphism,
			where $n$ is the product of the residual characteristic exponents of $Z$.
		\item \label{main: item: perfect residue field case}
			$\eta_{M}$ is an isomorphism
			if the residue fields of $Z$ are perfect.
	\end{enumerate}
\end{MainThm}

\eqref{main: item: trivial duality} is more or less trivial
(akin to the adjunction $j^{\ast} \leftrightarrow j_{\ast}$ or
$j_{!} \leftrightarrow j^{\ast}$).
Therefore the real content of duality is the three equivalent statements
in \eqref{main: item: equivalence},
which is a generalization of Grothendieck's duality conjecture.
\eqref{main: item: semistable case} easily reduces to
Grothendieck's duality conjecture for semistable abelian varieties proved in \cite{Wer97}.
For \eqref{main: item: after inverting res char},
we define $l$-adic realizations of $\mathcal{N}(M)$ and $\mathcal{N}_{0}(M)$ (resp.\ $\mathcal{P}(M)$)
as constructible complexes of sheaves of $\Z_{l}$-modules on $X$ (resp.\ $Z$),
where $l$ is a prime invertible on $Z$,
and use the six operations formalism (in particular, duality) in $l$-adic derived categories.
\eqref{main: item: perfect residue field case} generalizes the result of \cite{Suz14} for abelian varieties.
We will prove \eqref{main: item: perfect residue field case}
using the duality for cohomology of local fields with perfect residue fields with coefficients in $M$
that is established in \cite[Thm.\ (9.1)]{Suz14}.


\subsection{Remarks and organization}
\label{sec: Remarks and organization}

Here are some remarks.
If $U = \Spec K$ and $X$ has infinitely many points,
then $\mathcal{N}_{0}(M) = [\mathcal{Y}_{0} \to \mathcal{G}_{0}]$ still makes sense;
see Def.\ \ref{def: extension by zero of etale group}
for the definition of the extension by zero $\mathcal{Y}_{0}$ in this setting.
But $\mathcal{Y}_{0}$ is not locally of finite type over $X$
since $\Spec K$ is not.
If one wants a duality in this case,
one should first extend $M$ to a $1$-motive over some dense open subscheme $V$ of $X$
and then consider the above duality for the morphism $V \into X$.

The target category $D^{b}(\SmGp / X)$ of the N\'eron model functor $\mathcal{N}$ is certainly not the best possible one.
In the current form, we cannot consider transitivity of N\'eron model functors
along two dense open subschemes $V \into U \into X$.
Also, an arbitrary object of $D^{b}(\SmGp / X)$ does not seem to have any meaningful notion of dual
such that the double dual recovers the original object.
For this reason, we do not attempt to lift the morphisms $\zeta_{M}$ and $\zeta_{0 M}$ to $D^{b}(\SmGp / X)$.
The correct target (resp.\ source) category might be a suitably defined (non-derived) category of
``constructible'' or even ``perverse'' $1$-motives over $X$ (resp.\ $U$),
and the functors $\mathcal{N}$ and $\mathcal{N}_{0}$ might be viewed as
$j_{\ast}$ and $j_{!}$ between such categories.

Other kinds of realizations of N\'eron models should be explored.
Among such would be the universal one after inverting the residual characteristics,
i.e.\ as mixed \'etale motives over $X$ in the sense of Cisinski-D\'eglise \cite{CD16}.
The answers to this and the previous questions might exist
along the lines of the work of Pepin Lehalleur \cite{PL15}.

The above duality results are essentially of local nature, reduced to each point of $Z$.
Global duality as studied in \cite[III, \S 3, 9, 11]{Mil06} and \cite{Suz18}
should be extended to N\'eron models of $1$-motives.

We will see in Prop.\ \ref{prop: Neron model from relative curve}
an example where the N\'eron model of a $1$-motive arises geometrically
from a relative curve over $X$ with an \'etale local section over $U$.
This suggests that N\'eron models of $1$-motives might have some role
in the study of rational points of curves over $K$ valued in ramified extensions of $K$
and the index problem for curves.

Now the organization of the paper is as follows.
In \S \ref{sec: Definition of Neron models},
after collecting some facts about representability of sheaves on the smooth site,
we define N\'eron models and connected N\'eron models,
thereby proving Thm.\ \ref{main: definition of Neron models}.
In \S \ref{sec: Component complexes},
we first study some generalities on morphisms of topologies without exact pullback functors,
such as the one $i \colon Z_{\sm} \to X_{\sm}$ and
the change of topologies $X_{\fppf} \to X_{\sm}$.
Then we define N\'eron component complexes.
In \S \ref{sec: Duals and duality pairings of Neron models},
we define the duality morphisms $\zeta_{M}$, $\zeta_{0 M}$ and $\eta_{M}$.
We prove Thm.\ \ref{main: duality}
\eqref{main: item: trivial duality},
\eqref{main: item: equivalence} and
\eqref{main: item: semistable case}.
We also prove a weaker version of \eqref{main: item: after inverting res char},
namely that $\eta_{M} \tensor \Q$ is an isomorphism,
by some arguments on connected-\'etale sequences.
In \S \ref{sec: l-adic realization and perfectness for l-part},
we define $l$-adic realizations and prove
Thm.\ \ref{main: duality} \eqref{main: item: after inverting res char}.
The weaker version of \eqref{main: item: after inverting res char} proved earlier
is necessary for this
since derived $l$-adic completions of semi-abelian varieties and lattices are
both $\Z_{l}$-lattices up to shift and destroy their distinction.
In \S \ref{sec: Duality for cohomology of Neron models and perfectness for p-part},
we prove Thm.\ \ref{main: duality} \eqref{main: item: perfect residue field case}.
We quickly recall the formalism of the ind-rational pro-\'etale site from \cite{Suz14}, \cite{Suz18}
and the duality result \cite[Thm.\ (9.1)]{Suz14}
on cohomology of local fields with perfect residue field with coefficients in $M$.
From this, we deduce its version for cohomology of the ring of integers of such a local field
with coefficients in $\mathcal{N}(M)$,
from which \eqref{main: item: perfect residue field case} follows.

\begin{Ack}
	The author thanks Kazuya Kato and Qing Liu for having helpful discussions.
	The author is also grateful to Cristian D.\ Gonz\'alez-Avil\'es
	for sharing his earlier research proposal containing his plan on developing
	Xavier Xarles's 1996 attempt for N\'eron models of $1$-motives,
	and to the referee for careful comments.
\end{Ack}

\begin{Notation}
	The categories of sets and abelian groups are denoted by
	$\Set$ and $\Ab$, respectively.
	All groups, group schemes and sheaves of groups are assumed commutative.
	For an additive category $\mathcal{A}$,
	the category of complexes in $\mathcal{A}$ in cohomological grading is denoted by $\Ch(\mathcal{A})$.
	Its full subcategories of bounded below, bounded above and bounded complexes are denoted by
	$\Ch^{+}(\mathcal{A})$, $\Ch^{-}(\mathcal{A})$ and $\Ch^{b}(\mathcal{A})$, respectively.
	If $A \to B$ is a morphism in $\Ch(\mathcal{A})$,
	then its mapping cone is denoted by $[A \to B]$.
	The homotopy category of $\Ch^{\bullet}(\mathcal{A})$ for $\bullet = +$, $-$, $b$ or (blank)
	is denoted by $K^{\bullet}(\mathcal{A})$.
	If $\mathcal{A}$ is abelian, then its derived category is denoted by
	$D^{\bullet}(\mathcal{A})$.
	The canonical truncation functors for $D(\mathcal{A})$ in degrees $\le n$ and $\ge n$ are denoted by
	$\tau_{\le n}$ and $\tau_{\ge n}$, respectively.
	If we say $A \to B \to C$ is a distinguished triangle in a triangulated category,
	we implicitly assume that a morphism $C \to A[1]$ to the shift of $A$ is given,
	and the triangle $A \to B \to C \to A[1]$ is distinguished.
	If $A \to B$ is a morphism in a triangulated category
	together with a certain canonical choice of a mapping cone,
	then this mapping cone is still denote by $[A \to B]$ unless confusion may occur.
	For a site $S$, the categories of sheaves of sets and abelian groups are denoted by
	$\Set(S)$ and $\Ab(S)$.
	We denote $\Ch^{\bullet}(S) = \Ch^{\bullet}(\Ab(S))$
	and use the notation $K^{\bullet}(S)$, $D^{\bullet}(S)$ similarly.
	The Hom and sheaf-Hom functors for $\Ab(S)$ are denoted by
	$\Hom_{S}$ and $\sheafhom_{S}$, respectively.
	Their right derived functors are denoted by
	$\Ext_{S}^{n}$, $R \Hom_{S}$ and $\sheafext_{S}^{n}$, $R \sheafhom_{S}$, respectively.
	The tensor product functor $\tensor$ is over the ring $\Z$
	(or, on some site, the sheaf of rings $\Z$).
	Its left derived functor is denoted by $\tensor^{L}$.
	For a morphism of sites $f \colon S' \to S$,
	we denote by $f^{\ast}$ the pullback functor for sheaves of abelian groups.
\end{Notation}


\section{Definition of N\'eron models}
\label{sec: Definition of Neron models}

For a scheme $X$,
we denote the smooth site of $X$ by $X_{\sm}$.
It is the category of smooth $X$-schemes with $X$-scheme morphisms
endowed with the \'etale (or equivalently, smooth) topology.
We denote the category of separated smooth group schemes
(commutative, as assumed throughout the paper) over $X$ by $\SmGp / X$
and the category of quasi-separated smooth (commutative!) group algebraic spaces over $X$ by $\SmGp' / X$.
They are additive categories.
The full subcategory of $\SmGp / X$ (resp.\ $\SmGp' / X$) consisting of objects \'etale over $X$
are denoted by $\EtGp / X$ (resp.\ $\EtGp' / X$).

By a Dedekind scheme, we mean a noetherian regular scheme of dimension $\le 1$.
A separated smooth group algebraic space over a Dedekind scheme is a scheme
by \cite[Thm.\ (3.3.1)]{Ray70a}.
Hence $\SmGp / X \subset \SmGp' / X$ if $X$ is Dedekind.

\begin{Prop} \label{prop: morphism between smooth sites of U and X}
	Let $X$ be an irreducible Dedekind scheme with function field $K$.
	Let $U$ be either a dense open subscheme of $X$ or equal to $\Spec K$.
	Then the inclusion morphism $j \colon U \into X$ induces a morphism of sites
	$j \colon U_{\sm} \to X_{\sm}$
	(defined by the functor sending a smooth $X$-scheme $X'$ to $X' \times_{X} U$).
\end{Prop}

\begin{proof}
	The only non-trivial part is the exactness of the pullback functor
	$j^{\ast \set} \colon \Set(X_{\sm}) \to \Set(U_{\sm})$
	for sheaves of sets.
	To show this, we may assume that $X = \Spec A$ is affine.
	If $U$ is open in $X$, then $j^{\ast \set}$ is just the restriction functor, hence exact.
	Assume $U = \Spec K$.
	Let $F \in \Set(X_{\sm})$.
	Then $j^{\ast \set} F$ is the sheafification of the presheaf
	that sends a smooth $K$-algebra $B$ to the direct limit of the sets $F(A')$,
	where $A'$ runs through smooth $A$-algebras with fixed $A$-algebra homomorphisms to $B$.
	The index category for this direct limit is filtered
	since $K$ and hence $B$ are filtered direct limits of smooth $A$-algebras.
	Since filtered direct limits and sheafification are exact,
	we know that $j^{\ast \set}$ is exact.
\end{proof}

In the rest of this section, assume the following:
\begin{Sit} \label{sit: to consider Neron models, U can be Spec K} \BetweenThmAndList
	\begin{itemize}
		\item
			$X$ is an irreducible Dedekind scheme with function field $K$.
		\item
			$U$ is either a dense open subscheme of $X$ or equal to $\Spec K$.
		\item
			$j \colon U \into X$ is the inclusion morphism.
		\item
			$j \colon U_{\sm} \to X_{\sm}$ is the morphism of sites induced by $j$
			as in Prop.\ \ref{prop: morphism between smooth sites of U and X}.
	\end{itemize}
\end{Sit}

As above, we assume that $X$ is irreducible (and, in particular, non-empty),
so that its function field $K$ makes sense.
A Dedekind scheme is a finite disjoint union of irreducible Dedekind schemes (\cite[1.1]{BLR90}).
The arguments in this paper do not involve with descent problems
that require careful treatment of reducible Dedekind schemes.
Note that $\Spec K \subset X$ is open
if and only if $\Spec K$ is (locally) of finite type over $X$
if and only if $X$ has finitely many points
if and only if $X$ has a finite open covering by local Dedekind schemes.

\begin{Prop} \label{prop: push of lattice is representable}
	Let $Y$ be a (possibly infinite) disjoint union of finite \'etale connected schemes over $U$.
	Then $j_{\ast} Y$ is a separated \'etale scheme over $X$.
\end{Prop}

\begin{proof}
	If $Z$ is a smooth $X$-scheme,
	then any connected component of $Z \times_{X} U$ uniquely extends
	to a connected component of $Z$.
	Therefore $j_{\ast}$ as a functor $\Set(U_{\sm}) \to \Set(X_{\sm})$ commutes with disjoint unions.
	Hence we may assume that $Y$ is connected.
	Let $\closure{Y}$ be the normalization of $X$ in $Y$.
	Let $V \subset \closure{Y}$ be the maximal open subscheme \'etale over $X$.
	If $Z$ is a smooth $X$-scheme,
	then any $U$-morphism $Z \times_{X} U \to Y$ uniquely extends
	to an $X$-morphism $Z \to \closure{Y}$ since $Z$ is normal.
	This morphism factors through $V$.
	This means that $j_{\ast} Y = V$,
	which is separated \'etale.
\end{proof}

\begin{Prop} \label{prop: kernel from etale is etale}
	Let $Y$ be an \'etale group scheme over $X$
	and $F \in \Ab(X_{\sm})$ a sheaf.
	Let $\varphi \colon Y \to F$ be any morphism in $\Ab(X_{\sm})$
	and $\Ker(\varphi) \in \Ab(X_{\sm})$ its kernel.
	Then $\Ker(\varphi)$ is an open subscheme of $Y$ and,
	in particular, an \'etale $X$-scheme.
\end{Prop}

\begin{proof}
	Let $N$ be the union of the open subschemes of $Y$ that map to zero in $F$.
	Then $N$ itself maps to zero in $F$.
	Any $X$-morphism $Z \to Y$ from a quasi-compact smooth $X$-scheme $Z$
	is a faithfully flat smooth morphism followed by an open immersion.
	Hence if $Z$ maps to zero in $F$, then it factors through $N$.
	Thus $N = \ker(\varphi)$.
\end{proof}

\begin{Prop} \label{prop: extension of smooth by smooth is smooth}
	Let $G_{1}, G_{2} \in \SmGp' / X$.
	Then any extension $G_{3}$ of $G_{1}$ by $G_{2}$ in $\Ab(X_{\sm})$
	is in $\SmGp' / X$.
	If $G_{1}, G_{2} \in \SmGp / X$, then $G_{3} \in \SmGp / X$.
\end{Prop}

\begin{proof}
	We know that $G_{3} \in \SmGp' / X$ by descent.
	Since $X$ is Dedekind, we know by \cite[Thm.\ (3.3.1)]{Ray70a} that
	a separated group algebraic space over $X$ is a scheme.
	Hence the second statement follows.
\end{proof}

If $G$ is an extension of an abelian scheme by a torus over $U$ and if $U = \Spec K$,
then $j_{\ast} G$ is represented by the N\'eron (lft) model \cite[10.1/7]{BLR90},
which is in $\SmGp / X$.
If $U \subset X$ is dense open,
we still have $j_{\ast} G \in \SmGp / X$
by the arguments in \cite[10.1/9]{BLR90}.
In this case, $j_{\ast} G$ is the open subgroup scheme
of the N\'eron model of $G \times_{X} K$ along $\Spec K \to X$
with connected fibers over $U$.
We still call $j_{\ast} G$ the N\'eron model of $G$ (along $j \colon U \into X$).

\begin{Prop} \label{prop: push of lattice by semiabelian is representable}
	Let $0 \to H \to G \to Y \to 0$ be an extension of group schemes over $U$
	such that $H$ is an extension of an abelian scheme by a torus and $Y$ is a lattice.
	Then $j_{\ast} G \in \SmGp / X$.
\end{Prop}

\begin{proof}
	We have an exact sequence
	$0 \to j_{\ast} H \to j_{\ast} G \to j_{\ast} Y \to R^{1} j_{\ast} H$
	in $\Ab(X_{\sm})$.
	As above, we have $j_{\ast} H \in \SmGp / X$.
	Also $j_{\ast} Y \in \EtGp / X$
	by Prop.\ \ref{prop: push of lattice is representable}.
	By Prop.\ \ref{prop: kernel from etale is etale},
	we know that the kernel of $j_{\ast} Y \to R^{1} j_{\ast} H$ is in $\EtGp / X$.
	Therefore $j_{\ast} G \in \SmGp / X$
	by Prop.\ \ref{prop: extension of smooth by smooth is smooth}.
\end{proof}

Let $\mathcal{M}_{U}$ be the category of smooth $1$-motives over $U$
in the sense of Deligne \cite[(10.1.10)]{Del74}.
An object $M = [Y \to G]$ of $\mathcal{M}_{U}$ is a complex
consisting of a lattice $Y$ placed in degree $-1$,
an extension $G$ of an abelian scheme by a torus placed in degree $0$
and a morphism $Y \to G$ of group schemes over $U$.
A morphism in $\mathcal{M}_{U}$ is a morphism of complexes of group schemes over $U$.
In particular, $\mathcal{M}_{U}$ is a full subcategory of $\Ch^{b}(\SmGp / U)$.

For $M = [Y \to G] \in \mathcal{M}_{U}$ and a finite \'etale covering $V$ of $U$,
we denote the Weil restriction $\Res_{V / U}(Y \times_{U} V)$ of $Y \times_{U} V$ by $Y_{(V)}$.
For references on Weil restrictions, see \cite[7.6]{BLR90} and \cite[A.5]{CGP15}.
We have a natural injective morphism $Y \into Y_{(V)}$.
We set $Y_{(/ V)} = Y_{(V)} / Y$.
The objects $Y_{(V)}$ and $Y_{(/ V)}$ are lattices over $U$.
We denote the cokernel of the diagonal embedding $Y \into Y_{(V)} \oplus G$ in $\Ab(U_{\sm})$ by $G_{(V)}$
(which depends on not only $V$ and $G$ but the whole $M$ despite of the notation).
We have a morphism between exact sequences
	\[
		\begin{CD}
			0 @>>> Y @>>> Y_{(V)} @>>> Y_{(/ V)} @>>> 0 \\
			@. @VVV @VVV @| @. \\
			0 @>>> G @>>> G_{(V)} @>>> Y_{(/ V)} @>>> 0
		\end{CD}
	\]
in $\Ab(U_{\sm})$.
In particular, we have $G_{(V)} \in \SmGp / U$
by Prop.\ \ref{prop: extension of smooth by smooth is smooth}.
Define
	\[
			M_{(V)}
		=
			[Y_{(V)} \to G_{(V)}]
		\in
			\Ch^{b}(\SmGp / U).
	\]
Then the natural morphism $M \to M_{(V)}$ is a quasi-isomorphism in $\Ch^{b}(U_{\sm})$.
For a morphism $W \to V$ of finite \'etale coverings of $U$,
we have natural morphisms $Y_{(V)} \to Y_{(W)}$ and $G_{(V)} \to G_{(W)}$
by functoriality of the Weil restriction
and hence a morphism $M_{(V)} \to M_{(W)}$.

\begin{Def}
	Let $M = [Y \to G] \in \mathcal{M}_{U}$.
	A \emph{good covering} of $U$ with respect to $M$ (or $Y$) and $X$
	is a finite \'etale covering $V$ of $U$
	such that $Y \times_{U} V$ extends to a lattice over the normalization of $X$ in $V$.
\end{Def}

(We do not introduce a piece of notation for the above mentioned extension of $Y \times_{U} V$
as we do not have to.)
The key properties of good coverings are that
any finite \'etale covering that factors through a good covering is good
and that the following holds.

\begin{Prop} \label{prop: induced lattice has trivial first push}
	For any lattice $Y$ over $U$, a good covering exists.
	If $V$ is a good covering of $U$ with respect to $Y$ and $X$,
	then we have $R^{1} j_{\ast} Y_{(V)} = 0$.
\end{Prop}

\begin{proof}
	A lattice can be trivialized by a finite \'etale covering by
	\cite[X, Prop.\ 5.11, Thm.\ 5.16]{DG}.
	Such a covering is good.
	Let $V$ be a good covering of $U$ with respect to $Y$.
	Let $\closure{V}$ be the normalization of $X$ in $V$.
	Note that the Weil restriction functor $\Ab(V_{\sm}) \to \Ab(U_{\sm})$
	is nothing but the pushforward functor for the finite \'etale morphism $V \to U$.
	Since the pushforward functor for a finite morphism is exact in the \'etale topology (\cite[II, Cor.\ 3.6]{Mil80}),
	we know that the Weil restriction functor $\Ab(V_{\sm}) \to \Ab(U_{\sm})$ is exact
	(see also \cite[A.5.4]{CGP15}).
	Hence the sheaf $R^{1} j_{\ast} Y_{(V)} \in \Ab(X_{\sm})$ is the \'etale sheafification of the presheaf
	that sends a smooth $X$-scheme $X'$ to $H^{1}(X' \times_{X} V, Y)$.
	Hence it is enough to assume that $X$ is strict henselian local and
	$U$ is the generic point of $X$, and show that
	$H^{1}(X' \times_{X} V, Y) = 0$
	for the strict henselization $X'$ of any smooth $X$-scheme at any point.
	By goodness, $Y$ extends to a lattice over the strict henselian scheme $X' \times_{X} \closure{V}$.
	Hence $Y$ becomes trivial over $X' \times_{X} V$.
	Since $X' \times_{X} V$ is regular,
	we have $H^{1}(X' \times_{X} V, \Z) = 0$ and so $H^{1}(X' \times_{X} V, Y) = 0$.
\end{proof}

\begin{Def}
	Let $M = [Y \to G] \in \mathcal{M}_{U}$
	and $V$ a good covering of $U$ with respect to $M$ and $X$.
	We define
		\[
				\mathcal{N}(M, V)
			=
				j_{\ast} M_{(V)}
			=
				[j_{\ast} Y_{(V)} \to j_{\ast} G_{(V)}],
		\]
	which is an object of $\Ch^{b}(\SmGp / X)$
	by Prop.\ \ref{prop: push of lattice by semiabelian is representable}.
	We call $\mathcal{N}(M, V)$ the \emph{N\'eron model of $M$
	with respect to $V$} (along $j \colon U \to X$).
	The assignment $V \mapsto \mathcal{N}(M, V)$ is contravariantly functorial.
\end{Def}

\begin{Prop} \label{prop: Neron with V represents truncated push}
	Let $M = [Y \to G] \in \mathcal{M}_{U}$ and
	$V$ a good covering of $U$ with respect to $M$ and $X$.
	Consider the morphisms
		\[
				\mathcal{N}(M, V)
			\to
				R j_{\ast} M_{(V)}
			\cong
				R j_{\ast} M
		\]
	in $D^{b}(X_{\sm})$.
	Then the induced morphism
		\[
				\mathcal{N}(M, V)
			\to
				\tau_{\le 0} R j_{\ast} M
		\]
	is an isomorphism in $D^{b}(X_{\sm})$.
\end{Prop}

\begin{proof}
	By Prop.\ \ref{prop: induced lattice has trivial first push},
	we have an exact sequence
		\[
				0
			\to
				H^{-1} R j^{\ast} M
			\to
				j_{\ast} Y_{(V)}
			\to
				j_{\ast} G_{(V)}
			\to
				H^{0} R j^{\ast} M
			\to
				0.
		\]
	This proves the proposition.
\end{proof}

The inclusion functor $\SmGp' X \into \Ab(X_{\sm})$ induces a triangulated functor
$K^{b}(\SmGp' / X) \to K^{b}(X_{\sm})$.

\begin{Def} \label{def: localize category of groups}
	We say that a morphism in $K^{b}(\SmGp / X)$, $K^{b}(\SmGp' / X)$, $K^{b}(\EtGp / X)$ or $K^{b}(\EtGp' / X)$
	is a \emph{quasi-isomorphism} if it is so in $K^{b}(X_{\sm})$.
	We define $D^{b}(\SmGp / X)$ to be the localization of $K^{b}(\SmGp / X)$
	by the null system of objects quasi-isomorphic to zero (\cite[Thm.\ 10.2.3]{KS06}).
	We define $D^{b}(\SmGp' / X)$, $D^{b}(\EtGp / X)$ and $D^{b}(\EtGp' / X)$ similarly.
\end{Def}

A priori, $D^{b}(\SmGp / X)$ (and $D^{b}(\SmGp' / X)$) might not be a locally small category and
might be as large as $D^{b}(X_{\sm})$.
Smooth group schemes over $X$ with connected fibers are quasi-compact (\cite[Exp.\ VI B, Cor.\ 3.6]{DG70})
and hence form a small set.
Therefore the problem is the cardinalities of the component groups of the fibers.
If one wants to prove the local smallness of $D^{b}(\SmGp / X)$ using \cite[Rmk.\ 7.1.14]{KS06},
it suffices to show that for any quasi-isomorphism $F' \to F$ in $K^{b}(\SmGp / X)$,
there exists a quasi-isomorphism $F'' \to F'$ in $K^{b}(\SmGp / X)$
such that the cardinalities of the component groups of the fibers of the terms of $F''$ are no bigger than those of $F$.
While this seems likely verified by a limit argument on component groups,
we content ourselves with not necessarily locally small categories.
But notice that the relative component group $G / G_{0} = \pi_{0}(G)$
(the quotient by the relative identity component $G_{0}$;
\cite[Exp.\ VI B, Cor.\ 3.5]{DG70}, \cite[(3.2), d)]{Ray70a};
see also the first paragraph of the proof of Prop.\ \ref{prop: Neron components} below)
of any smooth group scheme and group algebraic space $G$ in this paper
are all $\Z$-constructible (\cite[II, \S 0, ``Constructible sheaves'']{Mil06}) over $X$.
Hence one may instead use the full subcategory of $\SmGp / X$ and $\SmGp' / X$
consisting of objects with $\Z$-constructible relative component groups.
This full subcategory is small, and hence any localization of its bounded homotopy category is small.

\begin{Prop}
	The natural functors induce a commutative diagram of triangulated functors
		\[
			\begin{CD}
					D^{b}(\EtGp / X)
				@>>>
					D^{b}(\EtGp' / X)
				@>>>
					D^{b}(X_{\et})
				\\
				@VVV
				@VVV
				@VVV
				\\
					D^{b}(\SmGp / X)
				@>>>
					D^{b}(\SmGp' / X)
				@>>>
					D^{b}(X_{\sm}).
			\end{CD}
		\]
\end{Prop}

\begin{proof}
	This follows from \cite[Thm.\ 10.2.3]{KS06}.
\end{proof}

\begin{Prop} \label{prop: term wise exact sequence gives dist triangle}
	Let $0 \to F_{1} \to F_{2} \to F_{3} \to 0$ be a term-wise exact sequence in $\Ch^{b}(X_{\sm})$
	with $F_{i} \in \Ch^{b}(\SmGp / X)$ for any $i$.
	Then there exists a canonical morphism $F_{3} \to F_{1}[1]$ in $D^{b}(\SmGp / X)$
	such that the triangle $F_{1} \to F_{2} \to F_{3} \to F_{1}[1]$ is distinguished in $D^{b}(\SmGp / X)$
	and maps to the canonical distinguished triangle
	$F_{1} \to F_{2} \to F_{3} \to F_{1}[1]$ in $D^{b}(X_{\sm})$.
	Similar statements hold for $D^{b}(\SmGp' / X)$, $D^{b}(\EtGp / X)$ and $D^{b}(\EtGp' / X)$.
\end{Prop}

\begin{proof}
	Set $F_{3}' = [F_{1} \to F_{2}]$.
	We have $F_{3}' = F_{2} \oplus F_{1}[1]$
	as a graded object forgetting the differentials.
	The first projection $F_{2} \oplus F_{1}[1] \to F_{2}$ followed by the morphism $F_{2} \to F_{3}$
	gives a morphism $F_{3}' \to F_{3}$ in $\Ch^{b}(\SmGp / X)$ (i.e.\ commutative with the differentials).
	For any $n$, the diagram with exact rows
		\[
			\begin{CD}
				H^{n} F_{1} @>>> H^{n} F_{2} @>>> H^{n} F_{3}' @>>> H^{n + 1} F_{1} @>>> H^{n + 1} F_{2} \\
				@| @| @VVV @| @| \\
				H^{n} F_{1} @>>> H^{n} F_{2} @>>> H^{n} F_{3} @>>> H^{n + 1} F_{1} @>>> H^{n + 1} F_{2}
			\end{CD}
		\]
	in $\Ab(X_{\sm})$ is commutative.
	Hence $F_{3}' \to F_{3}$ is a quasi-isomorphism.
	The required morphism is given by the composite $F_{3} \isomfrom F_{3}' \to F_{1}[1]$ in $D^{b}(\SmGp / X)$.
\end{proof}

\begin{Prop} \label{prop: functoriality of Neron models}
	For any object $M = [Y \to G] \in \mathcal{M}_{U}$,
	choose a good covering $V$ of $U$ with respect to $M$ and $X$
	and consider the object
		\[
				\mathcal{N}(M, V)
			\in
				D^{b}(\SmGp / X).
		\]
	For any morphism $M = [Y \to G] \to M' = [Y' \to G'] \in \mathcal{M}_{U}$,
	choose a good covering $V$ of $U$ with respect to both $M$ and $M'$ and $X$
	and consider the morphism
		\[
				\mathcal{N}(M, V)
			\to
				\mathcal{N}(M', V)
			\in
				D^{b}(\SmGp / X).
		\]
	These assignments define a well-defined additive functor
	$\mathcal{M}_{U} \to D^{b}(\SmGp / X)$.
\end{Prop}

\begin{proof}
	Let $M = [Y \to G] \in \mathcal{M}_{U}$ and
	$V$ a good covering of $U$ with respect to $M$ and $X$.
	If $W \to V$ is a morphism from another finite \'etale covering $W$ of $U$,
	then the induced morphism
	$\mathcal{N}(M, V) \to \mathcal{N}(M, W)$
	is a quasi-isomorphism by Prop.\ \ref{prop: Neron with V represents truncated push}.
	Let $f, g \colon W \rightrightarrows V$ be two $U$-morphisms.
	Then the diagram
		\[
			\begin{CD}
				0 @>>> Y @>>> Y_{(V)} @>>> Y_{(/ V)} @>>> 0 \\
				@. @VV 0 V @VV f - g V @VV f - g V @. \\
				0 @>>> Y @>>> Y_{(W)} @>>> Y_{(/ W)} @>>> 0
			\end{CD}
		\]
	is commutative.
	Hence the morphism $f - g \colon Y_{(V)} \to Y_{(W)}$ factors through the quotient $Y_{(/ V)}$.
	The composite $G_{(V)} \onto Y_{(/ V)} \to Y_{(W)}$ as a diagonal arrow in the commutative diagram
		\[
			\begin{CD}
				Y_{(V)} @>>> G_{(V)} \\
				@VV f - g V @V f - g VV \\
				Y_{(W)} @>>> G_{(W)}
			\end{CD}
		\]
	from the right upper term to the left lower term
	splits the diagram into two commutative triangles.
	This means that the two morphisms $f, g \colon M_{(V)} \rightrightarrows M_{(W)}$ are homotopic to each other.
	Applying $j_{\ast}$ term-wise,
	we know that the two morphisms
	$f, g \colon \mathcal{N}(M, V) \rightrightarrows \mathcal{N}(M, W)$
	are also homotopic to each other.
	Therefore $\mathcal{N}(M, V) \in D^{b}(\SmGp / X)$ is independent of the choice of $V$.
	The rest is an easy consequence of this.
\end{proof}

\begin{Def}
	We denote the functor $\mathcal{M}_{U} \to D^{b}(\SmGp / X)$
	defined in Prop.\ \ref{prop: functoriality of Neron models}
	by $\mathcal{N}$.
	Hence $\mathcal{N}(M) = \mathcal{N}(M, V)$ in $D^{b}(\SmGp / X)$
	for $M \in \mathcal{M}_{U}$ and a good covering $V$ of $U$ with respect to $M$ and $X$.
	We call $\mathcal{N}(M)$ the \emph{N\'eron model} of $M$ (along $j \colon U \to X$).
\end{Def}

\begin{Prop} \label{prop: Neron represents truncated push}
	For any $M \in \mathcal{M}_{U}$,
	the image of $\mathcal{N}(M)$ under the functor $D^{b}(\SmGp / X) \to D^{b}(X_{\sm})$
	is canonically identified with $\tau_{\le 0} R j_{\ast} M$.
\end{Prop}

\begin{proof}
	This follows from Prop.\ \ref{prop: Neron with V represents truncated push}.
\end{proof}

Prop.\ \ref{prop: functoriality of Neron models}
and \ref{prop: Neron represents truncated push} together
finish the proof of Thm.\ \ref{main: definition of Neron models}.
Next we consider connected N\'eron models.

\begin{Prop} \label{prop: extension by zero of etale group}
	Let $Y \in \EtGp / U$.
	Identify $U$ with its image in $Y$ as the zero section,
	which is an open and closed subset of $Y$.
	Then the scheme $(Y \setminus U) \bigsqcup X$ has a unique $X$-group scheme structure with zero section $X$
	compatible with the $U$-group scheme structure of $Y$.
	It is separated over $X$, and moreover \'etale over $X$
	if $U \subset X$ is open.
\end{Prop}

\begin{proof}
	Obvious.
\end{proof}

\begin{Def} \label{def: extension by zero of etale group}
	For $Y \in \EtGp / U$,
	we denote the group scheme $(Y \setminus U) \bigsqcup X$ over $X$
	in Prop.\ \ref{prop: extension by zero of etale group} by $\mathcal{Y}_{0}$
	and call it the \emph{extension by zero} of $Y$ to $X$.
\end{Def}

\begin{Prop}
	Let $Y \in \EtGp / U$ with extension by zero $\mathcal{Y}_{0}$ over $X$.
	Let $G$ be a group scheme over $X$.
	Then the natural homomorphism
		\[
				\Hom_{X}(\mathcal{Y}_{0}, G)
			\to
				\Hom_{U}(Y, G \times_{X} U)
		\]
	is an isomorphism.
	If $U \subset X$ is open, then we have $\mathcal{Y}_{0} = j_{!} Y$,
	where $j_{!} \colon \Ab(U_{\sm}) \to \Ab(X_{\sm})$ is the left adjoint of $j^{\ast}$.
\end{Prop}

\begin{proof}
	Obvious.
\end{proof}

\begin{Def}
	Let $M = [Y \to G] \in \mathcal{M}_{U}$.
	Let $\mathcal{Y}_{0} / X$ be the extension by zero of $Y$
	and $\mathcal{G}_{0} / X$ be the maximal open subgroup scheme of the N\'eron model $\mathcal{G}$ of $G$
	with connected fibers.
	We define
		\[
				\mathcal{N}_{0}(M)
			=
				[\mathcal{Y}_{0} \to \mathcal{G}_{0}],
		\]
	which is a complex of (not necessarily locally finite type) group schemes over $X$.
	The assignment $M \mapsto \mathcal{N}_{0}(M)$ is an additive functor.
	We call $\mathcal{N}_{0}(M)$ the \emph{connected N\'eron model} of $M$
	(along $j \colon U \to X$).
\end{Def}

We have $\mathcal{N}_{0}(M) \in \Ch^{b}(\SmGp / X)$ if $U \subset X$ is open.

\begin{Prop} \label{prop: connected Neron model}
	Let $M = [Y \to G] \in \mathcal{M}_{U}$.
	For good coverings $V$ of $U$ with respect to $M$, the natural morphisms
		\[
				\mathcal{N}_{0}(M)
			\to
				\mathcal{N}(M, V)
		\]
	of complexes of group schemes over $X$ are contravariantly functorial in $V$.
	In particular, if $U \subset X$ is open,
	then they induce a canonical morphism
		\[
				\mathcal{N}_{0}(M)
			\to
				\mathcal{N}(M)
			\text{ in }
				D^{b}(\SmGp / X).
		\]
\end{Prop}

\begin{proof}
	Obvious.
\end{proof}

In the next two propositions, consider the following situation:

\begin{Sit} \label{sit: base change}
		\[
			\begin{CD}
				U' @>> j' > X' \\
				@VVV @VVV \\
				U @> j >> X.
			\end{CD}
		\]
	is a cartesian diagram of schemes
	such that the both horizontal morphisms are as in
	Situation \ref{sit: to consider Neron models, U can be Spec K}.
\end{Sit}

In this situation,
we say that the formation of N\'eron models commutes with the base change $X' / X$
if the natural morphism
	\[
			\mathcal{N}(M, V) \times_{X} X'
		\to
			\mathcal{N}(M \times_{U} U', V \times_{U} U')
	\]
in $\Ch^{b}(\SmGp / X')$ is an isomorphism
for any $M \in \mathcal{M}_{U}$ and any good covering $V$ of $U$ with respect to $M$ and $X$.
In this case, the natural morphism
	\[
			\mathcal{N}_{0}(M) \times_{X} X'
		\to
			\mathcal{N}_{0}(M \times_{U} U')
	\]
of complexes of group schemes over $X'$ is an isomorphism
since the base change $(\var) \times_{X} X'$ preserves the maximal subgroup scheme with connected fibers.
The natural morphism
	\[
			\mathcal{N}(M) \times_{X} X'
		\to
			\mathcal{N}(M \times_{U} U')
	\]
in $D^{b}(\SmGp / X')$ is also an isomorphism.

\begin{Prop} \label{prop: Neron model commutes with regular base change}
	The formation of N\'eron models commutes with the base change $X' / X$
	if $X' \to X$ is a regular morphism.
	This happens, in particular, if $X' \to X$
	is an \'etale morphism, the localization of $X$ at a closed point,
	or the (strict) henselization of local $X$.
\end{Prop}

\begin{proof}
	The statement holds if $X' \to X$ is an open immersion.
	The statement is Zariski local on $X$ and $X'$.
	Hence we may assume that both $X$ and $X'$ are affine.
	
	By the structure of $\mathcal{N}(M, V)$,
	it is enough to show that
	$(j_{\ast} T) \times_{X} X' \isomto j'_{\ast}(T \times_{U} U')$
	for any smooth $U$-scheme $T$ such that
	$j_{\ast} T$ (resp.\ $j'_{\ast}(T \times_{U} U')$) is representable by
	a smooth $X$-scheme (resp.\ smooth $X'$-scheme).
	Since $X' \to X$ is a regular morphism between noetherian affine schemes,
	we know by Popescu's theorem \cite[Thm.\ 1.1]{Swa98} that
	$X'$ can be written as a filtered inverse limit $\invlim X'_{\lambda}$ of smooth affine $X$-schemes.
	Let $X''$ be a smooth affine $X'$-scheme.
	Then there exist an index $\lambda_{0}$ and
	a smooth affine $X'_{\lambda_{0}}$-scheme $X''_{\lambda_{0}}$
	such that $X'' \cong X' \times_{X'_{\lambda_{0}}} X''_{\lambda_{0}}$.
	Set $X''_{\lambda} = X'_{\lambda} \times_{X'_{\lambda_{0}}} X''_{\lambda_{0}}$
	for $\lambda \ge \lambda_{0}$.
	Then
		\begin{align*}
			&
					\Gamma(X'', (j_{\ast} T) \times_{X} X')
				=
					\dirlim_{\lambda \ge \lambda_{0}}
					\Gamma(X''_{\lambda} \times_{X} U, T)
			\\
			&	=
					\Gamma(X'' \times_{X} U, T)
				=
					\Gamma(X'', j'_{\ast}(T \times_{U} U')).
		\end{align*}
	Hence $(j_{\ast} T) \times_{X} X' \isomto j'_{\ast}(T \times_{U} U')$.
\end{proof}

\begin{Prop} \label{prop: Neron model commutes with completion}
	The formation of N\'eron models commutes with the base change $X' / X$
	if $X = \Spec A$ is local and $X' = \Spec \Hat{A}$ its completion.
\end{Prop}

\begin{proof}
	We may assume that $X$ is strictly henselian (local!) and $U = \Spec K$ its generic point.
	Write $U' = \Spec K'$.
	We denote the base change $(\var) \times_{X} X'$ of $X$-schemes by $(\var)'$.
	Let $V$ be a good covering of $U$ with respect to $M$ and $X$.
	The proof of Prop.\ \ref{prop: push of lattice is representable} shows that
	$j_{\ast} Y$ is the maximal open subscheme of the normalization of $X$ in $Y$
	\'etale over $X$.
	This description shows that $(j_{\ast} Y)' \isomto j'_{\ast} Y'$.
	Similarly, we have
	$(j_{\ast} Y_{(V)})' \isomto j'_{\ast} Y'_{(V')}$ and
	$(j_{\ast} Y_{(/ V)})' \isomto j'_{\ast} Y'_{(/ V')}$.
	Also, the formation of N\'eron (lft) model of semi-abelian varieties
	commutes with completion by \cite[10.1/3]{BLR90}.
	Thus $(j_{\ast} G)' \isomto j'_{\ast} G'$.
	We have exact sequences
		\begin{gather*}
					0
				\to
					j_{\ast} G
				\to
					j_{\ast} G_{(V)}
				\to
					j_{\ast} Y_{(/ V)}
				\to
					R^{1} j_{\ast} G,
			\\
					0,
				\to
					j_{\ast}' G'
				\to
					j_{\ast}' G'_{(V')}
				\to
					j_{\ast}' Y'_{(/ V')}
				\to
					R^{1} j_{\ast}' G'
		\end{gather*}
	in $\Ab(X_{\sm})$, $\Ab(X'_{\sm})$, respectively,
	and a commutative diagram with exact rows
		\[
			\begin{CD}
					0
				@>>>
					(j_{\ast} G)'
				@>>>
					(j_{\ast} G_{(V)})'
				@>>>
					(j_{\ast} Y_{(/ V)})'
			\\
			@.
			@VV \wr V
			@VVV
			@VV \wr V
			\\
					0
				@>>>
					j_{\ast}' G'
				@>>>
					j_{\ast}' G'_{(V')}
				@>>>
					j_{\ast}' Y'_{(/ V')}
			\end{CD}
		\]
	in $\SmGp / X'$.
	We want to show that the second vertical morphism is an isomorphism.
	Let $C \in \EtGp / X$ be the kernel of $j_{\ast} Y_{(/ V)} \to R^{1} j_{\ast} G$
	and $D \in \EtGp / X'$ the kernel of $j_{\ast}' Y'_{(/ V')} \to R^{1} j'_{\ast} G'$.
	The above diagram induces an injective morphism $C' \into D$ in $\EtGp / X'$,
	which is an isomorphism after $j'^{\ast}$.
	It is enough to show that $C' \isomto D$,
	for which it is enough to show that $\Gamma(X', C') \isomto \Gamma(X', D)$.
	We have $\Gamma(X', C') = \Gamma(X, C)$ since $C \in \EtGp / X$.
	We have a commutative diagram with exact rows
		\[
			\begin{CD}
					0
				@>>>
					\Gamma(X, C)
				@>>>
					\Gamma(U, Y_{(/ V)})
				@>>>
					H^{1}(U, G)
				\\
				@.
				@VVV
				@VV \wr V
				@VVV
				\\
					0
				@>>>
					\Gamma(X', D)
				@>>>
					\Gamma(U', Y'_{(/ V')})
				@>>>
					H^{1}(U', G').
			\end{CD}
		\]
	Since $G$ is smooth and $X$ henselian local,
	the right vertical homomorphism $H^{1}(U, G) \to H^{1}(U', G')$
	(or $H^{1}(K, G) \to H^{1}(K', G')$) is an isomorphism
	by \cite[Prop.\ 3.5.3 (2)]{GGMB14}.
	Thus $\Gamma(X, C) \isomto \Gamma(X', D)$.
	Hence $(j_{\ast} G_{(V)})' \isomto j_{\ast}' G'_{(V')}$.
	Therefore $\mathcal{N}(M, V)' \isomto \mathcal{N}(M', V')$.
\end{proof}

There are slightly more flexible representatives of $\mathcal{N}(M)$ in $\Ch^{b}(\SmGp / U)$
than $\mathcal{N}(M, V)$.

\begin{Prop} \label{prop: Neron model by flasque embedding}
	Let $M = [Y \to G] \in \mathcal{M}_{U}$.
	Let $0 \to Y \to Y' \to Y'' \to 0$ be an exact sequence of lattices over $U$
	such that $R^{1} j_{\ast} Y' = 0$.
	Denote the cokernel of the diagonal embedding $Y \into Y' \oplus G$ by $G'$.
	Then there exists a canonical isomorphism
	$\mathcal{N}(M) \cong [j_{\ast} Y' \to j_{\ast} G']$ in $D^{b}(\SmGp / X)$.
\end{Prop}

\begin{proof}
	Let $V$ be a finite \'etale covering of $U$
	such that $Y'$ (and hence $Y$ and $Y''$) is trivial over $V$.
	Choose a retraction (left-inverse) $Y' \times_{U} V \onto Y \times_{U} V$
	to the inclusion $Y \times_{U} V \into Y' \times_{U} V$.
	Such a retraction exists since the cokernel $Y'' \times_{U} V$ is a lattice.
	This retraction corresponds to a morphism $Y' \to Y_{(V)}$
	such that the composite $Y \to Y' \to Y_{(V)}$ is the natural inclusion.
	Hence we have a morphism $[Y' \to G'] \to [Y_{(V)} \to G_{(V)}]$ in $\Ch^{b}(\SmGp / U)$.
	We have $[j_{\ast} Y' \to j_{\ast} G'] \cong \tau_{\le 0} R j_{\ast} M$ in $D^{b}(X_{\sm})$
	since $R^{1} j_{\ast} Y' = 0$.
	Hence $[j_{\ast} Y' \to j_{\ast} G'] \to [j_{\ast} Y_{(V)} \to j_{\ast} G_{(V)}]$ in $\Ch^{b}(\SmGp / X)$
	is a quasi-isomorphism.
	Thus $[j_{\ast} Y' \to j_{\ast} G'] \isomto \mathcal{N}(M)$ in $D^{b}(\SmGp / X)$.
	A different choice of a retraction $Y' \times_{U} V \onto Y \times_{U} V$
	gives a morphism $[Y' \to G'] \to [Y_{(V)} \to G_{(V)}]$
	homotopic to the previous one
	by the same argument as the proof of Prop.\ \ref{prop: functoriality of Neron models}.
	Hence the isomorphism $[j_{\ast} Y' \to j_{\ast} G'] \isomto \mathcal{N}(M)$ in $D^{b}(\SmGp / X)$ is canonical.
\end{proof}

Here is a simple example where the N\'eron model of a $1$-motive arises geometrically
from a relative curve over $X$ with an \'etale local section over $U$.

\begin{Prop} \label{prop: Neron model from relative curve}
	Assume that $X$ is excellent and the residue fields of $X \setminus U$ are perfect.
	Let $S \to X$ be a proper flat morphism with $1$-dimensional geometrically connected fibers from a regular scheme $S$
	such that $S_{U} = S \times_{X} U \to U$ is smooth.
	Let $T \to X$ be a finite flat morphism from a regular scheme $T$
	such that $T_{U} = T \times_{X} U \to U$ is \'etale.
	Let $s \colon T \to S$ be an $X$-morphism.
	
	Denote by $Y$ the kernel of the norm map $\Res_{T_{U} / U} \Z \onto \Z$ and
	by $A$ the relative Jacobian $\Pic_{S_{U} / U}^{0}$.
	Let $\Res_{T_{U} / U} \Z \to \Pic_{S_{U} / U}$ be the morphism in $\SmGp / U$
	induced by the restriction $T_{U} \to S_{U}$ of $s$
	and $Y \to A$ its restriction.
	Set $M = [Y \to A] \in \mathcal{M}_{U}$.
	Denote by $(\Pic_{S / X})_{\sep}$ the maximal separated quotient of $\Pic_{S / X}$
	(\cite[(8.0.1)]{Ray70a}).
	
	Then the morphism $\Res_{T_{U} / U} \Z \to \Pic_{S_{U} / U}$ in $\SmGp / U$
	uniquely extends to a morphism
	$\Res_{T / X} \Z \to (\Pic_{S / X})_{\sep}$ in $\SmGp / X$,
	and we have a canonical isomorphism
		\[
				\mathcal{N}(M)
			\cong
				[\Res_{T / X} \Z \to (\Pic_{S / X})_{\sep}]
		\]
	in $D^{b}(\SmGp / X)$.
\end{Prop}

\begin{proof}
	We have a commutative diagram with exact rows
		\[
			\begin{CD}
					0
				@>>>
					Y
				@>>>
					\Res_{T_{U} / U} \Z
				@>>>
					\Z
				@>>>
					0
				\\
				@. @VVV @VVV @| @.
				\\
					0
				@>>>
					A
				@>>>
					\Pic_{S_{U} / U}
				@> \deg >>
					\Z
				@>>>
					0
			\end{CD}
		\]
	in $\SmGp / U$.
	Hence the cokernel of the diagonal embedding
	$Y \into \Res_{T_{U} / U} \Z \oplus A$ is $\Pic_{S_{U} / U}$.
	We have $R^{1} j_{\ast} \Res_{T_{U} / U} \Z = 0$
	by the same argument as the proof of
	Prop.\ \ref{prop: induced lattice has trivial first push}.
	Hence we have a canonical isomorphism
		\[
				\mathcal{N}(M)
			\cong
				\bigl[
						j_{\ast} \Res_{T_{U} / U} \Z
					\to
						j_{\ast} \Pic_{S_{U} / U}
				\bigr]
		\]
	in $D^{b}(\SmGp / X)$
	by Prop.\ \ref{prop: Neron model by flasque embedding}.
	We have $j_{\ast} \Res_{T_{U} / U} \Z \cong \Res_{T / X} \Z$
	since $\Res_{T / X} \Z$ is an \'etale $X$-scheme
	and the pushforward of $\Z$ by the morphism $T_{U, \et} \into T_{\et}$ is $\Z$.
	By the assumption on the residue fields of $X \setminus U$ and \cite[Eq.\ (4.10 bis)]{Gro68},
	we know that the natural morphism
	$\Pic_{S / X} \to j_{\ast} \Pic_{S_{U} / U}$ is surjective in $\Ab(X_{\et})$
	whose kernel is a skyscraper \'etale sheaf.
	Hence $(\Pic_{S / X})_{\sep} \isomto j_{\ast} \Pic_{S_{U} / U}$ in $\SmGp / X$.
	This proves the proposition.
\end{proof}

The assumption ``the residue fields of $X \setminus U$ are perfect'' is only used
to ensure that $\Pic_{S / X} \to j_{\ast} \Pic_{S_{U} / U}$ is surjective in $\Ab(X_{\et})$.
This latter condition is satisfied also if
$S \times_{X} \Order_{X, x}^{\mathit{sh}} \to \Spec \Order_{X, x}^{\mathit{sh}}$
has a section (or slightly weaker, has index $1$) for any $x \in X \setminus U$,
where $\Order_{X, x}^{\mathit{sh}}$ is the strict henselian local ring at $x$,
as stated before \cite[Eq.\ (4.13)]{Gro68}.


\section{Component complexes}
\label{sec: Component complexes}

Recall from \cite[\S 2.4]{Suz18} that a premorphism of sites $f \colon S' \to S$
between sites defined by pretopologies is a functor $f^{-1}$
from the underlying category of $S$ to the underlying category of $S'$
sending covering families to covering families
such that $f^{-1}(T_{2} \times_{T_{1}} T_{3}) = f^{-1} T_{2} \times_{f^{-1} T_{1}} f^{-1} T_{3}$
whenever $T_{2} \to T_{1}$ appears in a covering family.
Such a functor $f^{-1}$ is called a morphism of topologies from $S$ to $S'$
in \cite[Def.\ 2.4.2]{Art62}.
By \cite[Lem.\ 3.7.2]{Suz13},
the pullback $f^{\ast} \colon \Ab(S) \to \Ab(S')$ admits a left derived functor
$L f^{\ast} \colon D(S) \to D(S')$,
which is left adjoint to $R f_{\ast} \colon D(S') \to D(S)$.
Statement \eqref{item: sheafified derived push pull adjunction} in the following proposition
is already used in the proof of \cite[Lem.\ 3.2.6]{Suz18}.

\begin{Prop} \label{prop: premorphisms preserving finite products}
	Let $f \colon S' \to S$ be a premorphism of sites defined by pretopologies.
	Assume that the underlying category of $S$ has finite products and
	the underlying functor $f^{-1}$ of $f$ commutes with these products.
	\begin{enumerate}
		\item \label{item: sheafified derived push pull adjunction}
			There exists a canonical isomorphism
				\[
						R f_{\ast} R \sheafhom_{S'}(L f^{\ast} F, F')
					\cong
						R \sheafhom_{S}(F, R f_{\ast} F')
				\]
			in $D(S)$ functorial in $F \in D(S)$ and $F' \in D(S')$.
		\item \label{item: derived pull of representables}
			Denote the sheafification functor for $S$ or $S'$ by $(\var)^{\sim}$.
			Let $F$ be a bounded above complex of representable presheaves of abelian groups on $S$.
			Then we have $L f^{\ast} \Tilde{F} = f^{\ast} \Tilde{F} = (f^{-1} F)^{\sim}$ in $D(S')$,
			where $f^{\ast}$ and $f^{-1}$ in the middle and right-hand sides are applied term-wise.
	\end{enumerate}
\end{Prop}

\begin{proof}
	First note that the pullback $f^{\ast \set} \colon \Set(S) \to \Set(S')$ for sheaves of sets
	commutes with finite products.
	Indeed, for $F \in \Set(S)$, the sheaf $f^{\ast \set} F$ is the sheafification of the presheaf
	that sends $X' \in S'$ to the direct limit of $F(X)$,
	where $X$ runs through objects of $S$ together with morphisms $X' \to f^{-1} X$ in $S'$.
	Using this presentation and the assumption on $f^{-1}$,
	it is routine to check that
	the assignment $F \mapsto f^{\ast \set} F$ commutes with finite products.
	(If one wants a reference, see \cite[Thm.\ 1.5, Ex.\ 3.1]{BD77}.)
	
	Second, if $F$ is a complex of representable presheaves of abelian groups on $S$,
	then $f^{\ast} \Tilde{F} = f^{\ast \set} \Tilde{F} = (f^{-1} F)^{\sim}$ in $\Ch(S')$.
	
	\eqref{item: sheafified derived push pull adjunction}
	This is \cite[Thm.\ 18.6.9 (iii)]{KS06} when $f$ is a morphism of sites.
	The only part that needs exactness of $f^{\ast \set}$ is the proof of \cite[Prop.\ 17.6.7 (i)]{KS06}.
	It does not need full exactness but only commutativity with finite products.
	
	\eqref{item: derived pull of representables}
	We have a spectral sequence
		\[
					E_{1}^{i j}
				=
					L_{-j} f^{\ast} \Tilde{F}^{i}
			\Longrightarrow
				H^{i + j} L f^{\ast} \Tilde{F},
		\]
	where $\Tilde{F}^{i}$ is the $i$-th term of the complex $\Tilde{F}$.
	Hence we may assume that $F$ has a term only in degree zero.
	By the method of proof of \cite[Rmk.\ (5.1.2)]{Suz14} (i.e.\ using Mac Lane's resolution of $F$),
	the statement to prove reduces to the statement
	$f^{\ast \set}(\Tilde{F}^{m}) = ((f^{\ast} F)^{\sim})^{m}$
	for all $m \ge 0$,
	where the upper scripts $m$ denote products of $m$ copies.
	This statement is true by the assumption on $f^{-1}$
	and the two remarks above.
\end{proof}

\begin{Prop} \label{prop: derived pull of smooth groups}
	Let $f \colon Y \to X$ be a morphism of schemes
	and $f \colon Y_{\sm} \to X_{\sm}$ the induced premorphism of sites.
	Then for any bounded above complex $F$ of smooth group algebraic spaces over $X$,
	we have $L f^{\ast} F = F \times_{X} Y$ (term-wise fiber product).
\end{Prop}

\begin{proof}
	Consider the category of smooth algebraic spaces over $X$ with 
	morphisms of algebraic spaces over $X$
	endowed with the \'etale topology.
	Let $X_{\sm'}$ be the resulting site.
	The identity functor defines a morphism of sites
	$X_{\sm'} \to X_{\sm}$ inducing an equivalence on the topoi.
	Let $f' \colon Y_{\sm'} \to X_{\sm'}$ be the premorphism of sites induced by $f$.
	We have $L f'^{\ast} = L f^{\ast}$ since these functors are intrinsic to the topoi.
	Since the terms of $F$ are now representable in $X_{\sm'}$,
	Prop.\ \ref{prop: premorphisms preserving finite products} shows that
	$L f^{\ast} F = f^{\ast} F = F \times_{X} Y$.
\end{proof}

In the rest of this paper, we consider the following situation:
\begin{Sit} \label{sit: to consider Neron models, U open and Z closed} \BetweenThmAndList
	\begin{itemize}
		\item
			$X$ is an irreducible Dedekind scheme with function field $K$.
		\item
			$U$ is a dense open subscheme of $X$
			with complement $Z$ with reduced induced structure.
		\item
			$j \colon U \into X$ and $i \colon Z \into X$ are the inclusion morphisms.
		\item
			$j \colon U_{\sm} \to X_{\sm}$ and $i \colon Z_{\sm} \to X_{\sm}$
			are the premorphisms of sites induced by $j$ and $i$, respectively.
	\end{itemize}
\end{Sit}
The scheme $Z$ is a finite set of closed points of $X$.
Note that we disallow $U = \Spec K$ from now on (if $X$ has infinitely many points).
Hence the connected N\'eron model $\mathcal{N}_{0}(M)$ of $M \in \mathcal{M}_{U}$
is in $\Ch^{b}(\SmGp / X)$.
Let $\EtGpf / Z \subset \EtGp / Z$ be the full subcategory of groups with finitely generated geometric fibers.
For an object $P$ of $D^{b}(\EtGpf / Z)$,
we denote its linear dual $R \sheafhom_{Z_{\et}}(P, \Z) \in D^{b}(\EtGpf / Z)$ by $P^{\LDual}$.
Note that if $P$ is not constant, then $i_{\ast} P \in \EtGp' / X$ is only an algebraic space and not a scheme
by \cite[Prop.\ (3.3.6.1)]{Ray70a} (see also \cite[Introduction, Example 2]{Knu71}).

\begin{Prop} \label{prop: Neron components}
	Let $M = [Y \to G] \in \mathcal{M}_{U}$
	and $V$ a good covering of $U$ with respect to $M$ and $X$.
	Let $\mathcal{Y}_{0} / X$ be the extension by zero of $Y$
	and $\mathcal{G}_{0} / X$ be the maximal open subgroup scheme of the N\'eron model $\mathcal{G}$ of $G$
	with connected fibers.
	Define
		\begin{gather*}
					\mathcal{P}'(M, V)
				=
					[j_{\ast} Y_{(V)} / \mathcal{Y}_{0} \to j_{\ast} G_{(V)} / \mathcal{G}_{0}]
				\in
					\Ch^{b}(X_{\sm}),
			\\
					\mathcal{P}(M, V)
				=
					i^{\ast} \mathcal{P}'(M, V)
				\in
					\Ch^{b}(Z_{\sm}).
		\end{gather*}
	Then we have $\mathcal{P}'(M, V) \in \Ch^{b}(\EtGp' / X)$
	and $\mathcal{P}(M, V) \in \Ch^{b}(\EtGp / Z)$.
	We have a distinguished triangle
		\[
				\mathcal{N}_{0}(M)
			\to
				\mathcal{N}(M, V)
			\to
				\mathcal{P}'(M, V)
		\]
	in $D^{b}(\SmGp' / X)$.
	The morphism $\mathcal{P}'(M, V) \to i_{\ast} \mathcal{P}(M, V)$ is a quasi-isomorphism.
	In particular, we have a distinguished triangle
		\[
				\mathcal{N}_{0}(M)
			\to
				\mathcal{N}(M, V)
			\to
				i_{\ast} \mathcal{P}(M, V)
		\]
	in $D^{b}(\SmGp' / X)$.
	For any morphism $W \to V$ of finite \'etale coverings of $U$,
	the morphisms $\mathcal{P}'(M, V) \to \mathcal{P}'(M, W)$
	and $\mathcal{P}'(M, V) \to \mathcal{P}'(M, W)$ are both quasi-isomorphisms.
\end{Prop}

\begin{proof}
	We have $j_{\ast} G = \mathcal{G}$.
	We show that $\mathcal{G} / \mathcal{G}_{0} \in \EtGp' / X$.
	Write $\mathcal{G}$ as a union of quasi-compact open subschemes $S_{i} \subset G$
	(which might not be group subschemes).
	For each $i$, the connected components of fibers of $S_{i} \to X$
	form a quasi-separated \'etale algebraic space $\pi_{0}(S_{i} / X)$ over $X$
	(\cite[(6.8.1) (i)]{LMB00}, \cite[Thm.\ 2.5.2 (i)]{Rom11}).
	Since $\dirlim_{i} \pi_{0}(S_{i} / X) = \mathcal{G} / \mathcal{G}_{0}$,
	we know that the sheaf of groups $\mathcal{G} / \mathcal{G}_{0}$ is
	also a quasi-separated \'etale algebraic space,
	so it is in $\EtGp' / X$.
	
	We have an exact sequence
		\begin{equation} \label{eq: zeroth term of Neron components}
				0
			\to
				j_{\ast} G / \mathcal{G}_{0}
			\to
				j_{\ast} G_{(V)} / \mathcal{G}_{0}
			\to
				j_{\ast} Y_{/ (V)}
			\to
				R^{1} j_{\ast} G
		\end{equation}
	in $\Ab(X_{\sm})$.
	The kernel of $j_{\ast} Y_{/ (V)} \to R^{1} j_{\ast} G$ is in $\EtGp / X$
	by Prop.\ \ref{prop: push of lattice is representable}
	and \ref{prop: kernel from etale is etale}.
	Hence $j_{\ast} G_{(V)} / \mathcal{G}_{0} \in \EtGp' / X$.
	The same argument shows that
	$j_{\ast} Y_{(V)} / \mathcal{Y}_{0} \in \EtGp' / X$.
	Hence $\mathcal{P}'(M, V) \in \Ch^{b}(\EtGp' / X)$
	and consequently $\mathcal{P}(M, V) \in \Ch^{b}(\EtGp / Z)$.
	
	The sequence
		\[
				0
			\to
				\mathcal{N}_{0}(M)
			\to
				\mathcal{N}(M, V)
			\to
				\mathcal{P}'(M, V)
			\to
				0
		\]
	is a term-wise exact sequence of complexes in $\Ab(X_{\sm})$.
	Hence by Prop.\ \ref{prop: term wise exact sequence gives dist triangle},
	it defines a distinguished triangle in $D^{b}(\SmGp' / X)$.
	
	We have a distinguished triangle
		\[
				j_{!} j^{\ast} \mathcal{P}'(M, V)
			\to
				\mathcal{P}'(M, V)
			\to
				i_{\ast} i^{\ast} \mathcal{P}'(M, V)
		\]
	in $D^{b}(X_{\et})$ and hence in $D^{b}(X_{\sm})$.
	We have
		\[
				j^{\ast} \mathcal{P}'(M, V)
			=
				[Y_{(V)} / Y \to G_{(V)} / G]
			=
				0
		\]
	in $D^{b}(U_{\sm})$.
	Hence $\mathcal{P}'(M, V) \to i_{\ast} i^{\ast} \mathcal{P}'(M, V)$ is a quasi-isomorphism.
	
	We have a morphism of distinguished triangles
		\[
			\begin{CD}
					\mathcal{N}_{0}(M)
				@>>>
					\mathcal{N}(M, V)
				@>>>
					\mathcal{P}'(M, V)
				\\
				@| @VVV @VVV
				\\
					\mathcal{N}_{0}(M)
				@>>>
					\mathcal{N}(M, W)
				@>>>
					\mathcal{P}'(M, W).
			\end{CD}
		\]
	(Remember that there are hidden shifted terms $\mathcal{N}_{0}(M)[1]$ in the triangles
	and a hidden commutative square next to the right square.)
	The middle vertical morphism is a quasi-isomorphism
	by Prop.\ \ref{prop: Neron with V represents truncated push}.
	Hence so is the right vertical one.
\end{proof}

\begin{Def}
	For $M = [Y \to V] \in \mathcal{M}_{U}$
	and $V$ a good covering of $U$ with respect to $M$ and $X$,
	we call $\mathcal{P}(M, V) \in \Ch^{b}(\EtGp / Z)$
	the \emph{N\'eron component complex of $M$ with respect to $V$}.
	It is contravariantly functorial in $V$.
\end{Def}

\begin{Prop} \label{prop: fibers of Nerom components with trivialization}
	For $M$ and $V$ as above,
	we have $\mathcal{P}(M, V) \in \Ch^{b}(\EtGpf / Z)$.
	Its term in degree $-1$ is a lattice.
\end{Prop}

\begin{proof}
	The N\'eron model $j_{\ast} G$ has finitely generated groups of geometric connected components
	(\cite[Prop.\ 3.5]{HN11}).
	Hence the exact sequence \eqref{eq: zeroth term of Neron components} shows that
	$\mathcal{P}'(M, V)$ has finitely generated geometric fibers.
	Its degree $-1$ term $j_{\ast} Y_{(V)} / \mathcal{Y}_{0}$
	becomes $i^{\ast} j_{\ast} Y_{(V)}$ after pulling back to $Z$,
	which is a lattice.
\end{proof}

\begin{Prop}
	For any object $M = [Y \to G] \in \mathcal{M}_{U}$,
	choose a good covering $V$ of $U$ with respect to $M$ and $X$
	and consider the object
		\[
				\mathcal{P}'(M, V)
			\in
				D^{b}(\EtGp' / X).
		\]
	For any morphism $M = [Y \to G] \to M' = [Y' \to G'] \in \mathcal{M}_{U}$,
	choose a good covering $V$ of $U$ with respect to both $M$ and $M'$ and $X$
	and consider the morphism
		\[
				\mathcal{P}'(M, V)
			\to
				\mathcal{P}'(M', V)
			\in
				D^{b}(\EtGp' / X).
		\]
	These assignments define a well-defined additive functor
	$\mathcal{M}_{U} \to D^{b}(\EtGp' / X)$.
	We denote this functor as $M \mapsto \mathcal{P}'(M)$.
\end{Prop}

\begin{proof}
	The same proof as Prop.\ \ref{prop: functoriality of Neron models} works.
\end{proof}

\begin{Def} \label{def: Neron components}
	For $M \in \mathcal{M}_{U}$,
	we define
		\[
				\mathcal{P}(M)
			=
				i^{\ast} \mathcal{P}'(M)
			\in
				D^{b}(\EtGp / Z)
		\]
	and call it the \emph{N\'eron component complex} of $M$.
	The assignment $M \mapsto \mathcal{P}(M)$ defines an additive functor
	$\mathcal{M}_{U} \to D^{b}(\EtGp / Z)$.
	We have a canonical distinguished triangle
		\[
				\mathcal{N}_{0}(M)
			\to
				\mathcal{N}(M)
			\to
				i_{\ast} \mathcal{P}(M)
		\]
	in $D^{b}(\SmGp' / X)$ coming from Prop.\ \ref{prop: Neron components}.
\end{Def}

Note that $L i^{\ast} \mathcal{P}'(M) = i^{\ast} \mathcal{P}'(M)$ in $D^{b}(Z_{\sm})$
by Prop.\ \ref{prop: derived pull of smooth groups}.
Hence the pullback functor $i^{\ast}$ in the above definition can be understood
in either the derived or non-derived sense.

\begin{Prop} \label{prop: fibers of Nerom components}
	We have $\mathcal{P}(M) \in D^{b}(\EtGpf / Z)$, concentrated in degrees $-1$ and $0$.
	Its $H^{-1}$ is a lattice.
\end{Prop}

\begin{proof}
	This follows from Prop.\ \ref{prop: fibers of Nerom components with trivialization}.
\end{proof}

\begin{Prop} \label{prop: Neron components commutes with base change}
	The formation of Neron component complexes
	commutes with regular base change and completion.
	More precisely, in Situation \ref{sit: base change},
	assume that $U \to X$ is open and $X' \to X$ is regular or the completion of a local $X$.
	Let $Z' = Z \times_{X} X'$.
	Then we have
		\[
				\mathcal{P}(M, V) \times_{Z} Z'
			\isomto
				\mathcal{P}(M \times_{U} U', V \times_{U} U')
		\]
	in $\Ch^{b}(\SmGp / Z')$ and
		\[
				\mathcal{P}(M) \times_{Z} Z'
			\isomto
				\mathcal{P}(M \times_{U} U')
		\]
	in $D^{b}(\SmGp / Z')$
	for any $M \in \mathcal{M}_{U}$ and any good covering $V$ of $U$ with respect to $M$ and $X$.
\end{Prop}

\begin{proof}
	This follows from Prop.\ \ref{prop: Neron model commutes with regular base change}
	and \ref{prop: Neron model commutes with completion}.
\end{proof}


\section{Duals and duality pairings of N\'eron models}
\label{sec: Duals and duality pairings of Neron models}

In the rest of this paper, we work in the following situation with a set of notation:

\begin{Sit} \label{sit: to consider Neron models, notation for M and its dual}
	\begin{itemize} \BetweenThmAndList
		\item
			$j \colon U_{\sm} \to X_{\sm}$, $i \colon Z_{\sm} \to X_{\sm}$ and $K$ are as in
			Situation \ref{sit: to consider Neron models, U open and Z closed}.
		\item
			$M = [Y \to G] \in \mathcal{M}_{U}$ is a $1$-motive,
			where $Y$ is a lattice and $G$ is an extension of an abelian scheme $A$ by a torus $T$.
		\item
			Let $\mathcal{T}$, $\mathcal{G}$, $\mathcal{A}$ be the N\'eron models of
			$T$, $G$, $A$, respectively, along $j$.
			Denote their open subgroup schemes with connected fibers by
			$\mathcal{T}_{0}$, $\mathcal{G}_{0}$, $\mathcal{A}_{0}$, respectively.
			Denote the extension by zero of $Y$ by $\mathcal{Y}_{0}$.
		\item
			Fibers over $Z$ are generally denoted by putting the subscript $(\var)_{Z}$.
			The notation $(\var)_{Z 0}$ or $(\var)_{0 Z}$ mean the identity component of $(\var)_{Z}$.
		\item
			$M^{\vee} = [Y' \to G']$ is the dual of $M$ (see below).
			The objects above corresponding to $M^{\vee}$ are denoted by putting primes $(\var)'$,
			such as $A'$, $\mathcal{T}_{0}'$, $\mathcal{G}_{0 Z}'$, etc.
	\end{itemize}
\end{Sit}

The dual $1$-motive $M^{\vee}$ \cite[(10.2.12), (10.2.13)]{Del74} of $M$
is equipped with a canonical $\Gm$-bi-extension $M^{\vee} \tensor^{L} M \to \Gm[1]$
as a morphism in $D(U_{\fppf})$ or $D(U_{\sm})$.
The induced morphism
	\begin{equation} \label{eq: duality for one motives over U}
			M^{\vee}
		\to
			\tau_{\le 0}
			R \sheafhom_{U_{\fppf}}(M, \Gm[1])
	\end{equation}
via the derived tensor-Hom adjunction (\cite[Thm.\ 18.6.4 (vii)]{KS06})
is an isomorphism in $D(U_{\fppf})$ (\cite[Thm.\ 1.11 (2)]{Jos13}).

\begin{Prop} \label{prop: fppf and smooth sites}
	For any scheme $S$,
	let $\alpha \colon S_{\fppf} \to S_{\sm}$ be the premorphism of sites defined by the identity functor.
	Then for any bounded complex of smooth group algebraic spaces $F$ over $S$,
	we have $R \alpha_{\ast} F = F$ and $L \alpha^{\ast} F = F$.
	Let $F_{1}, F_{2}, F_{3}$ be bounded complexes of smooth group algebraic spaces over $S$.
	If $F_{1} \to F_{2} \to F_{3} \to F_{1}[1]$ is a triangle of bounded complexes of smooth group algebraic spaces over $S$,
	then it is distinguished in $D(S_{\fppf})$ if and only if it is so in $D(S_{\sm})$.
	Morphisms $F_{1} \tensor^{L} F_{2} \to F_{3}$ in $D(S_{\fppf})$
	bijectively correspond to morphisms $F_{1} \tensor^{L} F_{2} \to F_{3}$ in $D(S_{\sm})$.
\end{Prop}

\begin{proof}
	We have $R \alpha_{\ast} F = F$ since fppf cohomology with coefficients in smooth group algebraic spaces
	agrees with \'etale cohomology (\cite[III, Rmk.\ 3.11 (b)]{Mil80}).
	We have $L \alpha^{\ast} F = F$ by
	Prop.\ \ref{prop: premorphisms preserving finite products}
	\eqref{item: derived pull of representables}
	and the proof of Prop.\ \ref{prop: derived pull of smooth groups}.
	These facts imply the statement about the distinguished triangle.
	We have
		\begin{align*}
					\Hom_{D(S_{\fppf})}(F_{1} \tensor^{L} F_{2}, F_{3})
			&	=
					\Hom_{D(S_{\fppf})}(L \alpha^{\ast} F_{1} \tensor^{L} L \alpha^{\ast} F_{2}, F_{3})
			\\
			&	=
					\Hom_{D(S_{\sm})}(F_{1} \tensor^{L} F_{2}, R \alpha_{\ast} F_{3})
			\\
			&	=
					\Hom_{D(S_{\sm})}(F_{1} \tensor^{L} F_{2}, F_{3}).
		\end{align*}
	This shows the last statement.
\end{proof}

\begin{Prop} \label{prop: dual one motive in smooth topology}
	The isomorphism \eqref{eq: duality for one motives over U} induces an isomorphism
		\[
				M^{\vee}
			\isomto
				\tau_{\le 0}
				R \sheafhom_{U_{\sm}}(M, \Gm[1])
		\]
	in $D(U_{\sm})$.
\end{Prop}

\begin{proof}
	Let $\alpha \colon U_{\fppf} \to U_{\sm}$ be the premorphism of sites defined by the identity functor.
	We apply $\tau_{\le 0} R \alpha_{\ast}$ to the mentioned isomorphism.
	By Prop.\ \ref{prop: fppf and smooth sites},
	we have $\tau_{\le 0} R \alpha_{\ast} M^{\vee} = \tau_{\le 0} M^{\vee} = M^{\vee}$ and
	$L \alpha^{\ast} M = M$.
	Also $\tau_{\le 0} R \alpha_{\ast} \tau_{\le 0} = \tau_{\le 0} R \alpha_{\ast}$
	since $\tau_{\le 0} R \alpha_{\ast} \tau_{\ge 1} = 0$.
	Hence
		\[
				\tau_{\le 0}
				R \alpha_{\ast}
				\tau_{\le 0}
				R \sheafhom_{U_{\fppf}}(M, \Gm[1])
			=
				\tau_{\le 0}
				R \sheafhom_{U_{\sm}}(M, \Gm[1])
		\]
	using Prop.\ \ref{prop: premorphisms preserving finite products}.
	This proves the proposition.
\end{proof}

\begin{Def} \label{def: dual of Neron model}
	We define
		\begin{gather*}
					\mathcal{N}(M)^{\vee}
				=
					\tau_{\le 0} R \sheafhom_{X_{\sm}}(\mathcal{N}(M), \Gm[1])
				\in
					D(X_{\sm}),
			\\
					\mathcal{N}_{0}(M)^{\vee}
				=
					\tau_{\le 0} R \sheafhom_{X_{\sm}}(\mathcal{N}_{0}(M), \Gm[1])
				\in
					D(X_{\sm}).
		\end{gather*}
\end{Def}

Let $\mathcal{G}_{m} / X$ be the N\'eron model of $\Gm / U$.
Recall from \cite[III, proof of Lem.\ C.10]{Mil06} that $R^{1} j_{\ast} \Gm = 0$.
Hence we have $\tau_{\le 0}(R j_{\ast} \Gm[1]) = \mathcal{G}_{m}[1]$.
We have a canonical exact sequence
	\[
		0 \to \Gm \to \mathcal{G}_{m} \to i_{\ast} \Z \to 0
	\]
in $\SmGp / X$.

\begin{Prop} \label{prop: trivial duality with Neron target}
	Consider the morphisms
		\[
				R j_{\ast} M^{\vee} \tensor^{L} R j_{\ast} M
			\to
				R j_{\ast} (M^{\vee} \tensor^{L} M)
			\to
				R j_{\ast} \Gm[1]
		\]
	and the induced morphism
		\[
				\mathcal{N}(M^{\vee}) \tensor^{L} \mathcal{N}(M)
			\to
				\mathcal{G}_{m}[1]
		\]
	in $D(X_{\sm})$.
	The two induced morphisms
		\begin{align*}
					\mathcal{N}(M^{\vee})
			&	\to
					\tau_{\le 0}
					R \sheafhom_{X_{\sm}}(\mathcal{N}(M), \mathcal{G}_{m}[1])
			\\
			&	\to
					\tau_{\le 0}
					R \sheafhom_{X_{\sm}}(\mathcal{N}_{0}(M), \mathcal{G}_{m}[1])
		\end{align*}
	are both isomorphisms in $D(X_{\sm})$.
\end{Prop}

\begin{proof}
	Applying $\tau_{\le 0} R j_{\ast}$ to the isomorphism in
	Prop.\ \ref{prop: dual one motive in smooth topology}
	and using $\tau_{\le 0} R j_{\ast} \tau_{\le 0} = \tau_{\le 0} R j_{\ast}$, we have
		\[
				\mathcal{N}(M^{\vee})
			\isomto
				\tau_{\le 0}
				R j_{\ast}
				R \sheafhom_{U_{\sm}}(M, \Gm[1]).
		\]
	We have $j^{\ast} \mathcal{N}(M) = j^{\ast} \mathcal{N}_{0}(M) = M$.
	Hence by Prop.\ \ref{prop: premorphisms preserving finite products}, we have
		\begin{align*}
					\mathcal{N}(M^{\vee})
			&	\isomto
					\tau_{\le 0}
					R \sheafhom_{X_{\sm}}(\mathcal{N}(M), R j_{\ast} \Gm[1])
			\\
			&	\isomto
					\tau_{\le 0}
					R \sheafhom_{X_{\sm}}(\mathcal{N}_{0}(M), R j_{\ast} \Gm[1]).
		\end{align*}
	Let $F = R \sheafhom_{X_{\sm}}(\mathcal{N}(M), \var)$ or
	$R \sheafhom_{X_{\sm}}(\mathcal{N}_{0}(M), \var)$.
	By the distinguished triangle
		\[
				F \tau_{\le 0}(R j_{\ast} \Gm[1])
			\to
				F(R j_{\ast} \Gm[1])
			\to
				F \tau_{\ge 1}(R j_{\ast} \Gm[1])
		\]
	and the fact $\tau_{\le 0}(R j_{\ast} \Gm[1]) = \mathcal{G}_{m}[1]$,
	we have
		\[
				\tau_{\le 0}
				F(\mathcal{G}_{m}[1])
			\isomto
				\tau_{\le 0} F(R j_{\ast} \Gm[1]).
		\]
	Hence the result follows.
\end{proof}

If $M$ is an abelian scheme,
then the isomorphisms in Prop.\ \ref{prop: trivial duality with Neron target}
agree with the isomorphism in \cite[III, Lem.\ C.10]{Mil06}
by construction.

\begin{Prop} \label{prop: commutativity of pairing on Neron}
	The morphism
		\[
				\mathcal{N}(M^{\vee}) \tensor^{L} \mathcal{N}(M)
			\to
				\mathcal{G}_{m}[1]
		\]
	in Prop.\ \ref{prop: trivial duality with Neron target} and the corresponding morphism
		\[
				\mathcal{N}(M^{\vee \vee}) \tensor^{L} \mathcal{N}(M^{\vee})
			\to
				\mathcal{G}_{m}[1]
		\]
	for $M^{\vee}$ are compatible under the biduality isomorphism $M \isomto M^{\vee \vee}$
	and switching the tensor factors.
\end{Prop}

\begin{proof}
	This follows from the fact that
	the morphism
		$
				M^{\vee} \tensor^{L} M
			\to
				\Gm[1]
		$
	and the corresponding morphism
		$
				M^{\vee \vee} \tensor^{L} M^{\vee}
			\to
				\Gm[1]
		$
	for $M^{\vee}$ are compatible under $M \isomto M^{\vee \vee}$
	and switching the tensor factors
	by the symmetric description of the pairings \cite[(10.2.12)]{Del74}.
\end{proof}

\begin{Prop} \label{prop: linear dual of connected Neron and trivial duality with Gm target}
	We have
		\[
				\tau_{\le 0}
				R \sheafhom_{X_{\sm}}(\mathcal{N}_{0}(M), i_{\ast} \Z[1])
			=
				0.
		\]
	Hence the isomorphism in Prop.\ \ref{prop: trivial duality with Neron target} induces a canonical isomorphism
		\[
				\mathcal{N}(M^{\vee})
			\cong
				\tau_{\le 0}
				R \sheafhom_{X_{\sm}}(\mathcal{N}_{0}(M), \Gm[1])
			=
				\mathcal{N}_{0}(M)^{\vee}.
		\]
\end{Prop}

\begin{proof}
	We have
		\[
				R \sheafhom_{X_{\sm}}(\mathcal{N}_{0}(M), i_{\ast} \Z[1])
			=
				i_{\ast}
				R \sheafhom_{Z_{\sm}}(i^{\ast} \mathcal{N}_{0}(M), \Z[1])
		\]
	by Prop.\ \ref{prop: premorphisms preserving finite products}
	and Prop.\ \ref{prop: derived pull of smooth groups}.
	We have
		$
				i^{\ast} \mathcal{N}_{0}(M)
			=
				\mathcal{G}_{Z 0}
		$
	since $i^{\ast} \mathcal{Y}_{0} = 0$.
	Hence
		\[
				\tau_{\le 0}
				R \sheafhom_{X_{\sm}}(\mathcal{N}_{0}(M), i_{\ast} \Z[1])
			=
				i_{\ast}
				\tau_{\le 0}
				R \sheafhom_{Z_{\sm}}(\mathcal{G}_{Z 0}, \Z[1]).
		\]
	The fibers of the group $\mathcal{G}_{Z 0}$ at the points of $Z$
	are connected smooth algebraic groups.
	Hence the same argument as \cite[III, paragraph after Lem.\ C 10]{Mil06} shows that
		\begin{gather*}
					\sheafhom_{Z_{\sm}}(\mathcal{G}_{Z 0}, \Z)
				=
					0,
			\\
					\sheafext_{Z_{\sm}}^{1}(\mathcal{G}_{Z 0}, \Z)
				=
					\sheafhom_{Z_{\sm}}(\pi_{0}(\mathcal{G}_{Z 0}), \Q / \Z)
				=
					0.
		\end{gather*}
	The result then follows.
\end{proof}

By the definition of $\mathcal{N}_{0}(M)^{\vee}$ and the derived tensor-Hom adjunction,
the isomorphism
	$
			\mathcal{N}(M^{\vee})
		\isomto
			\mathcal{N}_{0}(M)^{\vee}
	$
in Prop.\ \ref{prop: linear dual of connected Neron and trivial duality with Gm target}
gives a $\Gm$-bi-extension
	\[
			\mathcal{N}_{0}(M) \tensor^{L} \mathcal{N}(M^{\vee})
		\to
			\Gm[1]
	\]
as a morphism in $D(X_{\sm})$.
Switching $M$ and $M^{\vee}$ and applying the derived tensor-Hom adjunction,
we have a morphism
	\[
			\mathcal{N}_{0}(M^{\vee})
		\to
			R \sheafhom_{X_{\sm}}(\mathcal{N}(M), \Gm[1])
	\]
This morphism factors through the truncation $\tau_{\le 0}$ of the right-hand side
since $\mathcal{N}_{0}(M^{\vee})$ is concentrated in degrees $\le 0$.

\begin{Def} \label{def: duality morphisms for Neron}
	We denote the above obtained morphisms in $D(X_{\sm})$ by
		\begin{gather*}
					\zeta_{M}
				\colon
					\mathcal{N}(M^{\vee})
				\isomto
					\mathcal{N}_{0}(M)^{\vee},
			\\
					\zeta_{0 M}
				\colon
					\mathcal{N}_{0}(M^{\vee})
				\to
					\mathcal{N}(M)^{\vee}.
		\end{gather*}
\end{Def}

Hence Thm.\ \ref{main: duality} \eqref{main: item: trivial duality} has been proven
in Prop.\ \ref{prop: linear dual of connected Neron and trivial duality with Gm target}.
If $M$ is an abelian scheme,
then the morphism $\zeta_{0 M}$ agrees with the morphism in \cite[III, Lem.\ C.11]{Mil06}
by construction.

\begin{Prop} \label{prop: linear dual of Neron}
	We have
		\begin{align*}
					\tau_{\le 0}
					R \sheafhom_{X_{\sm}}(\mathcal{N}(M), i_{\ast} \Z[1])
			&	\isomfrom
					\tau_{\le 0}
					R \sheafhom_{X_{\sm}}(i_{\ast} \mathcal{P}(M), i_{\ast} \Z[1])
			\\
			&	=
					i_{\ast}
					\tau_{\le 0}
					R \sheafhom_{Z_{\sm}}(\mathcal{P}(M), \Z[1])
			\\
			&	\isomto
					i_{\ast} \mathcal{P}(M)^{\LDual}[1].
		\end{align*}
\end{Prop}

\begin{proof}
	The first isomorphism follows from
	Prop.\ \ref{prop: linear dual of connected Neron and trivial duality with Gm target}.
	The second equality is Prop.\ \ref{prop: premorphisms preserving finite products}.
	The third isomorphism follows from the fact that
	the category of abelian groups has projective dimension one,
	$\Ext^{1}(\Z, \var) = 0$
	and Prop.\ \ref{prop: fibers of Nerom components}.
\end{proof}

\begin{Prop} \label{prop: diagram of duality morphisms for Neron}
	Let us abbreviate the functor $R \sheafhom_{X_{\sm}}$ as $[\var, \var]_{X}$.
	Consider the two distinguished triangles
		\begin{gather*}
					\mathcal{N}_{0}(M^{\vee})
				\to
					\mathcal{N}(M^{\vee})
				\to
					i_{\ast} \mathcal{P}(M^{\vee}),
			\\
					\bigl[ \mathcal{N}(M), \Gm[1] \bigr]_{X}
				\to
					\bigl[ \mathcal{N}(M), \mathcal{G}_{m}[1] \bigr]_{X}
				\to
					\bigl[ \mathcal{N}(M), i_{\ast} \Z[1] \bigr]_{X}.
		\end{gather*}
	The morphism
		\[
				\mathcal{N}(M^{\vee})
			\to
				\bigl[ \mathcal{N}(M), \mathcal{G}_{m}[1] \bigr]_{X}
		\]
	in the middle coming from Prop.\ \ref{prop: trivial duality with Neron target}
	can uniquely be extended to a morphism of distinguished triangles
		\[
			\begin{CD}
					\mathcal{N}_{0}(M^{\vee})
				@>>>
					\mathcal{N}(M^{\vee})
				@>>>
					i_{\ast} \mathcal{P}(M^{\vee}),
				\\
				@VVV
				@VVV
				@VVV
				\\
					\bigl[ \mathcal{N}(M), \Gm[1] \bigr]_{X}
				@>>>
					\bigl[ \mathcal{N}(M), \mathcal{G}_{m}[1] \bigr]_{X}
				@>>>
					\bigl[ \mathcal{N}(M), i_{\ast} \Z[1] \bigr]_{X}.
			\end{CD}
		\]
	The left vertical morphism agrees with the morphism $\zeta_{0 M}$.
	The middle morphism becomes an isomorphism after truncation $\tau_{\le 0}$.
\end{Prop}

\begin{proof}
	If we show that any morphism
	from $\mathcal{N}_{0}(M^{\vee})$ or $\mathcal{N}_{0}(M^{\vee})[1]$
	to $\bigl[ \mathcal{N}(M), i_{\ast} \Z[1] \bigr]_{X}$ is necessarily zero,
	then the general lemma on triangulated categories \cite[Lem.\ (4.2.5)]{Suz14}
	gives the desired unique extension.
	Let $f$ be such a morphism.
	Then $f$ factors through
	the truncation $\tau_{\le 0}$ of $\bigl[ \mathcal{N}(M), i_{\ast} \Z[1] \bigr]_{X}$,
	which is isomorphic to $i_{\ast} \mathcal{P}(M)^{\LDual}[1]$
	by Prop.\ \ref{prop: linear dual of Neron}.
	By adjunction, $f$ corresponds to a morphism
	from $\mathcal{G}_{Z 0}'$ or $\mathcal{G}_{Z 0}'[1]$ to $\mathcal{P}(M)^{\LDual}[1]$.
	We have a spectral sequence
		\[
				E_{2}^{i j}
			=
				\sheafext_{Z_{\sm}}^{i}(H^{-j-1} \mathcal{P}(M), \Z)
			\Longrightarrow
				H^{i + j} \bigl(
					\mathcal{P}(M)^{\LDual}[1]
				\bigr).
		\]
	From this and by Prop.\ \ref{prop: fibers of Nerom components},
	we obtain an isomorphism and an exact sequence
		\begin{gather*}
					H^{-1} \bigl(
						\mathcal{P}(M)^{\LDual}[1]
					\bigr)
				=
					\sheafhom_{Z_{\sm}}(H^{0} \mathcal{P}(M), \Z),
			\\
					0
				\to
					\sheafext_{Z_{\sm}}^{1}(H^{0} \mathcal{P}(M), \Z)
				\to
					H^{0} \bigl(
						\mathcal{P}(M)^{\LDual}[1]
					\bigr)
				\to
					\sheafhom_{Z_{\sm}}(H^{-1} \mathcal{P}(M), \Z)
				\to
					0.
		\end{gather*}
	Hence $\mathcal{P}(M)^{\LDual}[1] \in D^{b}(\EtGpf / Z)$ is concentrated in degrees $-1$ and $0$
	whose $H^{-1}$ is a lattice.
	Therefore, about cohomology objects $H^{n}$ of $\mathcal{P}(M)^{\LDual}[1]$,
	we have
		\[
				\Hom_{Z_{\sm}}(\mathcal{G}_{Z 0}', H^{-1} \text{ or } H^{0})
			=
				\Ext_{Z_{\sm}}^{1}(\mathcal{G}_{Z 0}, H^{-1})
			=
				0,
		\]
	and $f$ has to be zero.
	
	On the other hand, we have a commutative diagram
		\[
			\begin{CD}
					\mathcal{N}_{0}(M^{\vee}) \tensor^{L} \mathcal{N}(M)
				@>>>
					\mathcal{N}(M^{\vee}) \tensor^{L} \mathcal{N}(M)
				\\
				@VVV
				@VVV
				\\
					\Gm[1]
				@>>>
					\mathcal{G}_{m}[1],
			\end{CD}
		\]
	where the right vertical morphism is the one in Prop.\ \ref{prop: trivial duality with Neron target}
	and the left vertical morphism is the one appearing in the definition of $\zeta_{0 M}$
	(in the paragraph before Def.\ \ref{def: duality morphisms for Neron}).
	Translating this using the derived tensor-Hom adjunction,
	we see that the left square in the statement of the proposition is also commutative
	if the morphism $\zeta_{0 M}$ is used in the left vertical morphism.
	Hence the left vertical morphism has to be $\zeta_{0 M}$ by uniqueness.
	The middle morphism becomes an isomorphism after truncation $\tau_{\le 0}$
	by Prop.\ \ref{prop: trivial duality with Neron target}.
\end{proof}

The objects in the upper row of the diagram in this proposition are concentrated in degrees $\le 0$.
Hence the third vertical morphisms factor as
	\[
			i_{\ast} \mathcal{P}(M^{\vee})
		\to
			i_{\ast} \mathcal{P}(M)^{\LDual}[1],
	\]
where we used the isomorphism in Prop.\ \ref{prop: linear dual of Neron}.
It is a morphism in $D(X_{\sm})$.
Pulling back by $i$, we have a morphism
	$
				\mathcal{P}(M^{\vee})
			\to
				\mathcal{P}(M)^{\LDual}[1]
	$
in $D^{b}(\EtGpf / Z)$.
By the derived tensor-Hom adjunction, this corresponds to a morphism
	\[
			\mathcal{P}(M^{\vee})
		\tensor^{L}
			\mathcal{P}(M)
		\to
			\Z[1]
	\]
in $D(Z_{\sm})$ (or $D(Z_{\et})$).

\begin{Def}
	We denote the above obtained morphism in $D^{b}(\EtGpf / Z)$ by
		\[
				\eta_{M}
			\colon
					\mathcal{P}(M^{\vee})
				\to
					\mathcal{P}(M)^{\LDual}[1].
		\]
\end{Def}

If $M$ is an abelian scheme,
then the morphism $\eta_{M}$ agrees with Grothendieck's pairing
by \cite[III, Lem.\ C.11]{Mil06}.

\begin{Prop} \label{prop: RHom from Neron components to Neron Gm}
	We have
		\[
				\tau_{\le 1}
				R \sheafhom_{X_{\sm}}(
					i_{\ast} \mathcal{P}(M),
					\mathcal{G}_{m}[1]
				)
			=
				0.
		\]
\end{Prop}

\begin{proof}
	We have
		\[
				R \sheafhom_{X_{\sm}}(
					i_{\ast} \mathcal{P}(M),
					R j_{\ast} \Gm[1]
				)
			=
				R j_{\ast}
				R \sheafhom_{U_{\sm}}(
					j^{\ast} i_{\ast} \mathcal{P}(M),
					\Gm[1]
				)
			=
				0
		\]
	by Prop.\ \ref{prop: premorphisms preserving finite products}
	\eqref{item: sheafified derived push pull adjunction}
	(or \cite[Thm. 18.6.9 (iii)]{KS06}, noting that $j$ is a morphism of sites).
	Together with $R^{1} j_{\ast} \Gm = 0$, we have
		\[
				R \sheafhom_{X_{\sm}}(
					i_{\ast} \mathcal{P}(M),
					\mathcal{G}_{m}[1]
				)
			=
				R \sheafhom_{X_{\sm}}(
					i_{\ast} \mathcal{P}(M),
					\tau_{\ge 2}
					R j_{\ast} \Gm
				).
		\]
	This is concentrated in degrees $\ge 2$.
\end{proof}

\begin{Prop} \label{prop: compatibility of two definitions of duality pairings}
	Consider the diagram
		\[
			\begin{CD}
					\bigl[
						\mathcal{N}(M),
						\mathcal{G}_{m}[1]
					\bigr]_{X}
				@>>>
					\bigl[
						\mathcal{N}_{0}(M),
						\mathcal{G}_{m}[1]
					\bigr]_{X}
				@<<<
					\bigl[
						\mathcal{N}_{0}(M),
						\Gm[1]
					\bigr]_{X}
				\\
				@VVV
				@.
				@VVV
				\\
					\bigl[
						\mathcal{N}(M),
						i_{\ast} \Z[1]
					\bigr]_{X}
				@<<<
					\bigl[
						i_{\ast} \mathcal{P}(M),
						i_{\ast} \Z[1]
					\bigr]_{X}
				@>>>
					\bigl[
						i_{\ast} \mathcal{P}(M),
						\Gm[1]
					\bigr]_{X}[1]
			\end{CD}
		\]
	of natural morphisms in $D(X_{\sm})$.
	Note that the mapping cones of the four horizontal morphisms are all concentrated in degrees $\ge 1$
	by Prop.\ \ref{prop: linear dual of connected Neron and trivial duality with Gm target}
	and \ref{prop: RHom from Neron components to Neron Gm}
	and hence the horizontal morphisms become isomorphisms after truncation $\tau_{\le 0}$.
	Then the resulting two morphisms
		\[
				\mathcal{N}_{0}(M)^{\vee}
			\rightrightarrows
				i_{\ast} \mathcal{P}(M)^{\LDual}[1]
		\]
	are equal, or equivalently, the above diagram becomes a commutative diagram after $\tau_{\le 0}$.
\end{Prop}

\begin{proof}
	Denote the distinguished triangles
		\begin{gather*}
				\mathcal{N}_{0}(M)
			\to
				\mathcal{N}(M)
			\to
				i_{\ast} \mathcal{P}(M),
			\\
				\Gm[1]
			\to
				\mathcal{G}_{m}[1]
			\to
				i_{\ast} \Z[1]
		\end{gather*}
	by $A \to B \to C$, $D \to E \to F$, respectively.
	Then the diagram can be written as
		\[
			\begin{CD}
					[B, E]_{X}
				@>>>
					[A, E]_{X}
				@<<<
					[A, D]_{X}
				\\
				@VVV
				@.
				@VVV
				\\
					[B, F]_{X}
				@<<<
					[C, F]_{X}
				@>>>
					[C, D]_{X}[1].
			\end{CD}
		\]
	As noted, $[A, F]_{X}$ and $[C, E]_{X}[1]$ are concentrated in degrees $\ge 1$.
	We want to show that the two morphisms
	$\tau_{\le 0} [A, E]_{X} \rightrightarrows \tau_{\le 0} [C, F]_{X}$ are equal.
	Let $T \in D(X_{\sm})$ be any object concentrated in degrees $\le 0$.
	Applying $\Hom_{X_{\sm}}(T, \var)$, we are comparing two homomorphisms
		\[
				\Hom_{X_{\sm}}(T \tensor^{L} A, E)
			\rightrightarrows
				\Hom_{X_{\sm}}(T \tensor^{L} C, F).
		\]
	Denote $T \tensor^{L} A$, $T \tensor^{L} B$, $T \tensor^{L} C$
	by $A'$, $B'$, $C'$, respectively.
	We want to show that the diagram
		\[
			\begin{CD}
					\Hom_{X_{\sm}}(B', E)
				@>> \sim >
					\Hom_{X_{\sm}}(A', E)
				@<< \sim <
					\Hom_{X_{\sm}}(A', D)
				\\
				@VVV
				@.
				@VVV
				\\
					\Hom_{X_{\sm}}(B', F)
				@< \sim <<
					\Hom_{X_{\sm}}(C', F)
				@> \sim >>
					\Hom_{X_{\sm}}(C', D[1])
			\end{CD}
		\]
	is commutative.
	We know that $R \Hom_{X_{\sm}}(A', F)$ and $R \Hom_{X_{\sm}}(C', E[1])$ are
	concentrated in degrees $\ge 1$.
	Let $f \in  \Hom_{X_{\sm}}(A', E)$.
	Sending $f$ to $\Hom_{X_{\sm}}(C', F)$ via the left side of the diagram,
	we have a commutative diagram
		\begin{equation} \label{eq: technical compatibility diagram}
			\begin{CD}
				A' @>>> B' @>>> C' @>>> A'[1] \\
				@VV f V @VVV @VVV @VV f V \\
				E @>>> F @>>> D[1] @>>> E[1]
			\end{CD}
		\end{equation}
	with diagonal arrows from $B'$ to $E$ and $C'$ to $F$
	splitting the left and middle squares into commutative triangles.
	(Note that the right square is automatically commutative
	since $\Hom_{X_{\sm}}(C', E[1]) = 0$.)
	Hence we have a commutative diagram
		\[
			\begin{CD}
				B' @>>> C' @>>> A'[1] @>>> B'[1] \\
				@VVV @VVV @. @VVV \\
				E @>>> F @>>> D[1] @>>> E[1].
			\end{CD}
		\]
	By an axiom of triangulated categories, there exists a morphism $A'[1] \to D[1]$
	that completes this diagram into a morphism of distinguished triangles.
	This morphism diagonally splits the right square of the diagram \eqref{eq: technical compatibility diagram}
	into commutative triangles.
	From this, we see that the two images of $f$ in $\Hom_{X_{\sm}}(C', F)$ are equal.
	This proves the proposition.
\end{proof}

\begin{Prop} \label{prop: morphism between connected etale triangle to dual such}
	The diagram
		\[
			\begin{CD}
					\mathcal{N}_{0}(M^{\vee})
				@>>>
					\mathcal{N}(M^{\vee})
				@>>>
					i_{\ast} \mathcal{P}(M^{\vee}),
				\\
				@VV \zeta_{0 M} V
				@VV \zeta_{M} V
				@VV d \compose i_{\ast} \eta_{M} V
				\\
					\bigl[ \mathcal{N}(M), \Gm[1] \bigr]_{X}
				@>>>
					\bigl[ \mathcal{N}_{0}(M), \Gm[1] \bigr]_{X}
				@>>>
					\bigl[ i_{\ast} \mathcal{P}(M), \Gm[1] \bigr]_{X}[1].
			\end{CD}
		\]
	is a morphism of distinguished triangles,
	where the $d$ in the right vertical morphism is the connecting morphism
	$i_{\ast} \Z[1] \to \Gm[2]$ of the triangle
	$\Gm \to \mathcal{G}_{m} \to i_{\ast} \Z$.
\end{Prop}

\begin{proof}
	It is enough to show the commutativity of the squares
	after applying $\tau_{\le 0}$ to the lower row
	since the upper row consists of objects concentrated in degrees $\le 0$.
	(Note that there are actually three squares whose commutativity has to be checked,
	one of which is hidden in the diagram.)
	Prop.\ \ref{prop: commutativity of pairing on Neron} and
	\ref{prop: compatibility of two definitions of duality pairings} show that
	the lower row after $\tau_{\le 0}$ can be identified with
	the lower row after $\tau_{\le 0}$ of the diagram in
	Prop.\ \ref{prop: diagram of duality morphisms for Neron}.
	This implies the result.
\end{proof}

\begin{Prop} \label{prop: symmetry of pairing on Neron components}
	Under the identification $M \isomto M^{\vee \vee}$, the dual
		\[
				\eta_{M}^{\LDual}
			\colon
				\mathcal{P}(M)
			\to
				\mathcal{P}(M^{\vee})^{\LDual}[1]
		\]
	of $\eta_{M}$ agrees with
		\[
				\eta_{M^{\vee}}
			\colon
				\mathcal{P}(M)
			\to
				\mathcal{P}(M^{\vee})^{\LDual}[1].
		\]
	In particular, $\eta_{M}$ is an isomorphism if and only if $\eta_{M^{\vee}}$ is so.
\end{Prop}

\begin{proof}
	Applying the derived tensor-Hom adjunction to the diagram in
	Prop.\ \ref{prop: morphism between connected etale triangle to dual such}
	and interchanging the tensor factors,
	we have a morphism of distinguished triangles
		\[
			\begin{CD}
					\mathcal{N}_{0}(M)
				@>>>
					\mathcal{N}(M)
				@>>>
					i_{\ast} \mathcal{P}(M),
				\\
				@VV \zeta_{0 M^{\vee}} V
				@VV \zeta_{M^{\vee}} V
				@VV d \compose i_{\ast} \eta_{M}^{\LDual} V
				\\
					\bigl[ \mathcal{N}(M^{\vee}), \Gm[1] \bigr]_{X}
				@>>>
					\bigl[ \mathcal{N}_{0}(M^{\vee}), \Gm[1] \bigr]_{X}
				@>>>
					\bigl[ i_{\ast} \mathcal{P}(M^{\vee}), \Gm[1] \bigr]_{X}[1].
			\end{CD}
		\]
	Using the uniqueness part of Prop.\ \ref{prop: diagram of duality morphisms for Neron},
	we know that $i_{\ast} \eta_{M}^{\LDual} = i_{\ast} \eta_{M^{\vee}}$.
\end{proof}

The following together with Prop.\ \ref{prop: symmetry of pairing on Neron components}
proves Thm.\ \ref{main: duality} \eqref{main: item: equivalence}.

\begin{Prop}
	The morphisms $\zeta_{0 M}$ and $\zeta_{0 M^{\vee}}$ are both isomorphisms
	if and only if $\eta_{M}$ is an isomorphism.
	If these equivalent conditions are satisfied,
	then the diagram in Prop.\ \ref{prop: morphism between connected etale triangle to dual such}
	induces an isomorphism of distinguished triangles
		\[
			\begin{CD}
					\mathcal{N}_{0}(M^{\vee})
				@>>>
					\mathcal{N}(M^{\vee})
				@>>>
					i_{\ast} \mathcal{P}(M^{\vee}),
				\\
				@V \wr V \zeta_{0 M} V
				@V \wr V \zeta_{M} V
				@V \wr V i_{\ast} \eta_{M} V
				\\
					\mathcal{N}(M)^{\vee}
				@>>>
					\mathcal{N}_{0}(M)^{\vee}
				@>>>
					i_{\ast} \mathcal{P}(M)^{\LDual}[1].
			\end{CD}
		\]
\end{Prop}

\begin{proof}
	Suppose that $\eta_{M}$ is an isomorphism.
	By Prop.\ \ref{prop: linear dual of connected Neron and trivial duality with Gm target}
	and Def.\ \ref{def: duality morphisms for Neron},
	the morphism $\zeta_{M} \colon \mathcal{N}(M^{\vee}) \to \mathcal{N}_{0}(M)^{\vee}$
	is an isomorphism.
	Then the five lemma applied to the diagram in
	Prop.\ \ref{prop: morphism between connected etale triangle to dual such}
	shows that $\zeta_{0 M} \colon \mathcal{N}_{0}(M^{\vee}) \to \mathcal{N}(M)^{\vee}$ is an isomorphism.
	On the other hand, Prop.\ \ref{prop: symmetry of pairing on Neron components} shows that
	$\eta_{M^{\vee}}$ is also an isomorphism.
	Hence the above argument applied to $M^{\vee}$ implies that
	$\zeta_{0 M^{\vee}}$ is an isomorphism.
	
	Conversely, suppose that $\zeta_{0 M}$ and $\zeta_{0 M^{\vee}}$ are isomorphisms.
	Then the same argument as above shows that
	$\eta_{M}$ and $\eta_{M^{\vee}}$ are isomorphisms in $H^{-1}$ and injective in $H^{0}$.
	Set $P = \mathcal{P}(M)$ and $P' = \mathcal{P}(M^{\vee})$.
	Denote the torsion part by $(\var)_{\tor}$
	and torsion-free quotient by $(\var)_{/ \tor}$.
	Then we have a commutative diagram with exact rows
		\[
			\begin{CD}
					0
				@>>>
					H^{0}(P')_{\tor}
				@>>>
					H^{0}(P')
				@>>>
					H^{0}(P')_{/ \tor}
				@>>>
					0
				\\
				@.
				@VVV
				@VVV
				@VVV
				@.
				\\
					0
				@>>>
					H^{0}(P)_{\tor}^{\PDual}
				@>>>
					H^{1}(P^{\LDual})
				@>>>
					H^{-1}(P)^{\LDual}
				@>>>
					0,
			\end{CD}
		\]
	where $H^{0}(P)_{\tor}^{\PDual}$ is a shorthand for $(H^{0}(P)_{\tor})^{\PDual}$.
	The middle vertical morphism is injective.
	The right vertical one is an isomorphism
	by Prop.\ \ref{prop: symmetry of pairing on Neron components}.
	These imply that the left vertical morphism is injective.
	Switching $M$ and $M^{\vee}$, we know that
	$H^{0}(P)_{\tor} \to H^{0}(P')_{\tor}^{\PDual}$ is also injective.
	As $H^{0}(P)_{\tor}$ and $H^{0}(P')_{\tor}$ are finite \'etale,
	we conclude that these injective morphisms are all isomorphisms.
	Therefore $\eta_{M}$ and $\eta_{M^{\vee}}$ are isomorphisms.
	
	The last statement about the diagram follows from the isomorphism
		\[
				d
			\colon
				\bigl[ i_{\ast} \mathcal{P}(M), i_{\ast} \Z[1] \bigr]_{X}
			\isomto
				\tau_{\le 0}
				\bigl[ i_{\ast} \mathcal{P}(M), \Gm[2] \bigr]_{X},
		\]
	which is a consequence of Prop.\ \ref{prop: RHom from Neron components to Neron Gm}.
\end{proof}

\begin{Prop} \label{prop: reduction to str hensel case}
	The morphism $\eta_{M}$ is an isomorphism
	if and only if the corresponding morphism
		\[
				\eta_{M \times_{U} K_{x}^{sh}}
			\colon
				\mathcal{P}(M^{\vee} \times_{U} K_{x}^{sh})
			\to
				\mathcal{P}(M \times_{U} K_{x}^{sh})^{\LDual}[1],
		\]
	for the strict henselian local field $K_{x}^{sh}$ at any point $x \in Z$
	is an isomorphism.
	We may replace $K_{x}^{sh}$ by its completion.
\end{Prop}

\begin{proof}
	This follows from Prop.\ \ref{prop: Neron components commutes with base change}.
\end{proof}

The following proves Thm.\ \ref{main: duality} \eqref{main: item: semistable case}.

\begin{Prop} \label{prop: Neron duality is true in semistable case}
	The morphism $\eta_{M}$ is an isomorphism
	if $M$ is semistable over $X$,
	i.e.\ if its torus and lattice parts are unramified over $X$
	and its abelian scheme part is semistable over $X$.
\end{Prop}

\begin{proof}
	We may assume that $X$ is strictly henselian local
	by Prop.\ \ref{prop: reduction to str hensel case}.
	If $M = \Gm$ or $\Z[1]$, then $\mathcal{P}(M) = \Z$ or $\Z[1]$, respectively,
	and $\eta_{M}$ is an isomorphism.
	Hence $\eta_{M}$ is an isomorphism
	if $M$ is an unramified torus or an unramified lattice shifted by one.
	If $M$ is a semistable abelian variety,
	then $\eta_{M}$ is Grothendieck's pairing,
	which is an isomorphism by \cite{Wer97}.
	
	Now we treat a general semistable $M$.
	Set $\mathcal{Y} = j_{\ast} Y$.
	By assumption, $X$ itself is a good covering of $X$.
	Hence $\mathcal{N}(M) = [\mathcal{Y} \to \mathcal{G}]$ and
	$\mathcal{N}_{0}(M) = [\mathcal{Y}_{0} \to \mathcal{G}_{0}]$.
	We have $R^{1} j_{\ast} T = 0$ since $T$ is a trivial torus
	and by the proof of \cite[III, Lem.\ C.10]{Mil06}.
	Hence we have an exact sequence
	$0 \to \mathcal{T} \to \mathcal{G} \to \mathcal{A} \to 0$.
	We have $\mathcal{G}_{0} \cap \mathcal{T} = \mathcal{T}_{0}$
	since $\mathcal{G}_{0}$ is of finite type
	and $\pi_{0}(\mathcal{T}_{Z})$ is torsion-free.
	Therefore we have an exact sequence
	$0 \to \mathcal{T}_{0} \to \mathcal{G}_{0} \to \mathcal{A}_{0} \to 0$.
	We have morphisms of distinguished triangles
		\[
			\begin{CD}
					\mathcal{T}_{0}'
				@>>>
					\mathcal{G}_{0}'
				@>>>
					\mathcal{A}_{0}'
				\\
				@VVV @VVV @VVV
				\\
					\bigl[ \mathcal{Y}[1], \Gm[1] \bigl]_{X}
				@>>>
					\bigl[ [\mathcal{Y} \to \mathcal{A}], \Gm[1] \bigl]_{X}
				@>>>
					\bigl[ \mathcal{A}, \Gm[1] \bigl]_{X}
			\end{CD}
		\]
	and
		\[
			\begin{CD}
					\mathcal{G}_{0}'
				@>>>
					[\mathcal{Y}_{0}' \to \mathcal{G}_{0}']
				@>>>
					\mathcal{Y}_{0}'[1]
				\\
				@VVV @VVV @VVV
				\\
					\bigl[ [\mathcal{Y} \to \mathcal{A}]), \Gm[1] \bigl]_{X}
				@>>>
					\bigl[ [\mathcal{Y} \to \mathcal{G}], \Gm[1] \bigl]_{X}
				@>>>
					\bigl[ \mathcal{T}, \Gm[1] \bigl]_{X}.
			\end{CD}
		\]
	In either diagram, the left and right vertical morphisms become isomorphisms
	after truncation $\tau_{\le 0}$ by the previously treated cases.
	Therefore the four lemmas imply that
	the middle vertical morphisms become isomorphisms
	after $\tau_{\le 0}$.
\end{proof}

The following proves a weaker version of
Thm.\ \ref{main: duality} \eqref{main: item: after inverting res char}.

\begin{Prop} \label{prop: duality holds rationally}
	The morphism $\eta_{M}$ becomes an isomorphism
	after tensoring with $\Q$.
\end{Prop}

\begin{proof}
	We may assume that $X$ is strictly henselian local and $U = \Spec K$
	by Prop.\ \ref{prop: reduction to str hensel case}.
	First we describe $\mathcal{P}(M) \tensor \Q$.
	We have $R^{1} j_{\ast} \Q = 0$
	since $H^{1}(X' \times_{X} U, \Q) = 0$
	for any quasi-compact smooth $X$-scheme $X'$
	by \cite[II, Lem.\ 2.10]{Mil06}.
	From this, by taking a finite Galois extension of $K$ that trivializes $Y$
	and arguing with a Hochschild-Serre spectral sequence,
	we know that $R^{1} j_{\ast} Y \tensor \Q = 0$.
	Hence $\mathcal{N}(M) \tensor \Q = [j_{\ast} Y \to j_{\ast} G] \tensor \Q$.
	Therefore $\mathcal{P}(M) \tensor \Q = [\Gamma(U, Y) \to \pi_{0}(\mathcal{G}_{Z})] \tensor \Q$,
	where we are viewing $\mathcal{P}(M)$ as a complex of abstract abelian groups
	since $X$ is strictly henselian and hence $Z$ is a geometric point.
	Let $F$ be the cokernel of $\mathcal{G}_{Z} \to \mathcal{A}_{Z}$ in $\Ab(Z_{\sm})$.
	Then $\Gamma(Z, F)$ is a subgroup of $H^{1}(U, T)$.
	If $K'$ is a finite Galois extension of $K$ that trivializes $T$,
	then the exact sequence
		\[
				0
			\to
				H^{1}(\Gal(K' / K), T(K'))
			\to
				H^{1}(K, T)
			\to
				\Gamma(\Gal(K' / K), H^{1}(K', T)),
		\]
	the vanishing $H^{1}(K', \Gm) = 0$ and \cite[VIII, \S 2, Cor.\ 1 to Prop.\ 4]{Ser79}
	show that $H^{1}(K, T)$ is killed by $[K' : K]$.
	Hence so is $\Gamma(Z, F)$.
	The snake lemma for the diagram
		\[
			\begin{CD}
					0
				@>>>
					\mathcal{G}_{Z 0} / \mathcal{T}_{Z 0}
				@>>>
					\mathcal{G}_{Z} / \mathcal{T}_{Z 0}
				@>>>
					\pi_{0}(\mathcal{G}_{Z})
				@>>>
					0
				\\
				@.
				@VVV
				@VVV
				@VVV
				@.
				\\
					0
				@>>>
					\mathcal{A}_{Z 0}
				@>>>
					\mathcal{A}_{Z}
				@>>>
					\pi_{0}(\mathcal{A}_{Z})
				@>>>
					0
			\end{CD}
		\]
	gives an exact sequence
		\[
				\mathcal{T}_{Z} \cap \mathcal{G}_{Z 0}
			\to
				\pi_{0}(\mathcal{T}_{Z})
			\to
				\Ker \bigl(
					\pi_{0}(\mathcal{G}_{Z}) \to \pi_{0}(\mathcal{A}_{Z})
				\bigr)
			\to
				\mathcal{A}_{Z 0} / \mathcal{G}_{Z 0}
			\to
				F
		\]
	in $\Ab(Z_{\sm})$.
	The first term is of finite type over $Z$.
	Hence the first morphism has finite image.
	The third morphism is a morphism from a finitely generated abelian group to a smooth algebraic group.
	The group of $Z$-valued points of the cokernel of this morphism is killed by $[K' : K]$.
	Such a morphism has finite image since $Z$ is a geometric point.
	The group $\pi_{0}(\mathcal{A}_{Z})$ is finite.
	Thus we have $\pi_{0}(\mathcal{T}_{Z}) \tensor \Q = \pi_{0}(\mathcal{G}_{Z}) \tensor \Q$.
	Therefore
		\[
				\mathcal{P}(M) \tensor \Q
			=
				[\Gamma(U, Y) \to \pi_{0}(\mathcal{T}_{Z})] \tensor \Q.
		\]
	
	Note that $Y$ and $T'$ are Cartier dual to each other.
	So are $Y'$ and $T$.
	The morphism $\eta_{M} \tensor \Q$ decomposes into two parts
		\begin{gather*}
					\Gamma(U, Y') \tensor \Q
				\to
					\Hom(\pi_{0}(\mathcal{T}_{Z}) \tensor \Q, \Q),
			\\
					\pi_{0}(\mathcal{T}_{Z}') \tensor \Q
				\to
					\Hom(\Gamma(U, Y) \tensor \Q, \Q)
		\end{gather*}
	given by
		\[
				\Gamma(U, Y')
			=
				\Hom_{U}(T, \Gm)
			\to
				\Hom_{Z}(\mathcal{T}_{Z}, \mathcal{G}_{m})
			\to
				\Hom(\pi_{0}(\mathcal{T}_{Z}), \Z)
		\]
	and the corresponding morphism for $Y$ and $T'$.
	It is a classical fact that the two parts above are isomorphisms.%
	\footnote{
		One way to quickly see this is the following.
		Let $l$ be a prime invertible on $Z$.
		The Kummer sequence gives $H^{1}(U, V_{l} T) = \pi_{0}(\mathcal{T}) \tensor \Q_{l}$,
		where $V_{l}$ is the rational $l$-adic Tate module of $T$.
		The $l$-adic representation $V_{l} T$ over $U$ is the Tate twist of the dual of $Y' \tensor \Q_{l}$.
		Hence the duality $H^{1}(U, V_{l} T) \leftrightarrow \Gamma(U, Y' \tensor \Q_{l})$
		of $l$-adic cohomology of strict henselian discrete valuation fields (\cite[Exp.\ I, Thm.\ 5.1]{Ill77}) gives the result.
	}
	Hence $\eta_{M} \tensor \Q$ is an isomorphism.
\end{proof}

\begin{Prop} \label{prop: duality holds up to finite etale skyscrapers}
	The morphism $\zeta_{0 M}$ induces an isomorphism in cohomologies in degrees $\ne 0$
	and an injection in $H^{0}$.
	The morphism $\eta_{M}$ induces an isomorphism in cohomologies in degrees $\ne -1, 0$
	and an injection in $H^{-1}$.
	We have an exact sequence
		\[
				0
			\to
				\Coker(H^{-1} \eta_{M})
			\to
				\Coker(H^{0} \zeta_{0 M})
			\to
				\Ker(H^{0} \eta_{M})
			\to
				0.
		\]
	Each of these three terms as well as $\Coker(H^{0} \eta_{M})$ is of the form $i_{\ast} N$
	for some finite \'etale group scheme $N$ over $Z$.
\end{Prop}

\begin{proof}
	The domains and codomains of the morphisms $\zeta_{0 M}$ and $\eta_{M}$ are concentrated
	in degrees $-1$, $0$.
	The diagram in Prop.\ \ref{prop: diagram of duality morphisms for Neron} shows that
	$H^{-1} \zeta_{0 M}$ is injective and induces exact sequences
		\begin{gather*}
					0
				\to
					\Coker(H^{-1} \zeta_{0 M})
				\to
					\Ker(H^{-1} \eta_{M})
				\to
					\Ker(H^{0} \zeta_{0 M})
				\to
					0,
			\\
					0
				\to
					\Coker(H^{-1} \eta_{M})
				\to
					\Coker(H^{0} \zeta_{0 M})
				\to
					\Ker(H^{0} \eta_{M})
				\to
					0.
		\end{gather*}
	All these groups are torsion by Prop.\ \ref{prop: duality holds rationally}.
	The group $\Ker(H^{-1} \eta_{M})$ is torsion-free by
	Prop.\ \ref{prop: fibers of Nerom components},
	hence zero.
	Therefore $\Coker(H^{-1} \zeta_{0 M}) = \Ker(H^{0} \zeta_{0 M}) = 0$.
	The groups $\Coker(H^{n} \eta_{M})$ and $\Ker(H^{n} \eta_{M})$ for any $n$ are of the stated form.
	Hence so is their extension $\Coker(H^{0} \zeta_{0 M})$.
\end{proof}

In particular, if $\zeta_{0 M}$ is an isomorphism,
then $\eta_{M}$ is an isomorphism in cohomologies of degrees $\ne 0$
and an injection with finite \'etale cokernel in $H^{0}$.
This might not imply that $\eta_{M}$ is an isomorphism
if no additional assumption is made on $\zeta_{0 M^{\vee}}$.
This point seems to exist already for the case that $M = A$ is an abelian scheme;
see \cite[Prop.\ 5.1]{Bos97}.
An injection of finite \'etale groups \emph{of the same order} is an isomorphism,
but it is not clear whether $\pi_{0}(\mathcal{A}_{x})$ and $\pi_{0}(\mathcal{A}'_{x})$
for $x \in Z$ have the same order or not.
According to Lorenzini \cite[Rmk.\ 7.1]{Lor17},
it is not known that the groups of geometric points of
$\pi_{0}(\mathcal{A}_{x})$ and $\pi_{0}(\mathcal{A}'_{x})$ are
abstractly isomorphic.


\section{$l$-adic realization and perfectness for $l$-part}
\label{sec: l-adic realization and perfectness for l-part}

We continue working in Situation \ref{sit: to consider Neron models, notation for M and its dual}.
Let $l$ be a prime number invertible on $X$.
Below we use the formalism of derived categories of $l$-adic sheaves given by \cite[\S 6]{BS15}, \cite[Tag 09C0]{Sta18}
for technical simplicity.
One can also use \cite{Eke90}.
We denote the pro-\'etale site of $X$ by $X_{\pro\et}$ (\cite[Def.\ 4.1.1]{BS15}).
The derived completeness and the derived completion (\cite[Lem.\ 3.4.9, Prop.\ 3.5.1 (3)]{BS15})
	\[
			\widehat{F}
		=
			R \invlim_{n}
			(F \tensor^{L} \Z / l^{n} \Z)
	\]
of an object $F \in D(X_{\pro\et})$ is
always taken with respect to the ideal $l \Z \subset \Z$.
See \cite[Tag 091J]{Sta18}, \cite[\S 2.3]{Suz18}
for how we choose derived inverse limits functorially in derived categories.
In fact, we have
	\[
			\widehat{F}
		=
			R \sheafhom_{X_{\pro\et}}(\Q_{l} / \Z_{l}, F)[1]
	\]
by \cite[Tag 099B]{Sta18}.
The constructibility of a derived complete $F \in D(X_{\pro\et})$ is always taken
with respect to the ideal $l \Z \subset \Z$
(or $l \Z_{l} \subset \Z_{l}$; \cite[Lem.\ 3.5.6]{BS15})
unless otherwise noted.
The same notation applies to the pro-\'etale sites $U_{\pro\et}$, $Z_{\pro\et}$.
For a derived complete $F \in D(X_{\pro\et})$, we denote
	\[
			F^{\vee}
		=
			R \sheafhom_{X_{\pro\et}}(F, \Z_{l}(1)[2]),
	\]
where we set $\Z_{l} = \invlim_{n} \Z / l^{n} \Z \in \Ab(X_{\pro\et})$
and the Tate twist $\Z_{l}(1) = \invlim_{n} \Z / l^{n} \Z(1) \in \Ab(X_{\pro\et})$ as sheaves.
(This notation $F^{\vee}$ does not clash with dual $1$-motives $M^{\vee}$
since a non-zero $1$-motive is never derived complete.
It is also different from the linear dual of $\Z_{l}$-lattices
due to the twisted shift $(1)[2]$.)
The same notation $F^{\vee}$ applies to a derived complete $F \in D(U_{\pro\et})$.
For a derived complete $F \in D(Z_{\pro\et})$, we denote
	\[
			F^{\vee}
		=
			R \sheafhom_{Z_{\pro\et}}(F, \Z_{l}).
	\]
For the six operations formalism, see \cite[\S 6.7]{BS15}.
The $l$-adic Tate module $T_{l}(\var)$ of a sheaf is
the inverse limit of the $l^{n}$-torsion parts for $n \ge 0$.
Let $\nu \colon X_{\pro\et} \to X_{\et}$ be the morphism of sites
defined by the identity functor (\cite[\S 5]{BS15}).
We naturally regard objects of $D(X_{\et})$ as objects of $D(X_{\pro\et})$ via pullback $\nu^{\ast}$
(omitting $\nu^{\ast}$ from the notation).
A similar convention applies to $\nu \colon Z_{\pro\et} \to Z_{\et}$
and $\nu \colon U_{\pro\et} \to U_{\et}$.

\begin{Prop} \label{prop: ell adic realization is constructible}
	Let $\alpha \colon X_{\sm} \to X_{\et}$ be the morphism of sites defined by the identity.
	Then the objects $\alpha_{\ast} \mathcal{N}(M) \tensor^{L} \Z / l \Z$ and
	$\alpha_{\ast} \mathcal{N}_{0}(M) \tensor^{L} \Z / l \Z$ of $D^{b}(X_{\et})$
	are constructible complexes of sheaves of $\Z / l \Z$-modules.
\end{Prop}

\begin{proof}
	The object
		$
				\alpha_{\ast} \mathcal{Y}_{0} \tensor^{L} \Z / l \Z
			=
				\mathcal{Y}_{0} / l \mathcal{Y}_{0}
		$
	is constructible.
	The object
		$
				\alpha_{\ast} \mathcal{G}_{0} \tensor^{L} \Z / l \Z
		$
	is the $l$-torsion part of $\mathcal{G}_{0}$ shifted by one,
	which is constructible.
	Hence $\alpha_{\ast} \mathcal{N}_{0}(M) \tensor^{L} \Z / l \Z$ is constructible.
	Since $\mathcal{P}(M) \in D^{b}(\EtGpf / Z)$,
	it follows that $\alpha_{\ast} \mathcal{N}(M) \tensor^{L} \Z / l \Z$ is also constructible.
\end{proof}

In the rest of this section, we will omit $\alpha_{\ast}$ and
simply denote the image of a sheaf or a complex of sheaves $F$
over the smooth site (of $U$, $X$ or $Z$) by $F$.

\begin{Def}
	Viewing $\mathcal{N}(M)$ as an object of $D(X_{\pro\et})$,
	we call its derived completion $\mathcal{N}(M)^{\wedge} \in D(X_{\pro\et})$
	the \emph{$l$-adic realization} of $\mathcal{N}(M)$
	and denote it by $\widehat{\mathcal{N}}(M)$.
	The $l$-adic realizations $\widehat{\mathcal{N}}_{0}(M) = \mathcal{N}_{0}(M)^{\wedge} \in D(X_{\pro\et})$
	and $\widehat{\mathcal{P}}(M) = \mathcal{P}(M)^{\wedge} \in D(Z_{\pro\et})$
	are defined similarly.
\end{Def}

Yet another convention:
in the rest of this section,
we will use the pro-\'etale topology only
and denote the morphisms $U_{\pro\et} \to X_{\pro\et}$ and $Z_{\pro\et} \to Z_{\pro\et}$
induced by $j \colon U \into X$ and $i \colon Z \into X$ simply by $j$ and $i$.
This change of notation does not make a difference for relevant groups after derived completion.
More precisely:

\begin{Prop} \label{prop: sheaf pull and fiber after derived completion}
	Let $i \colon Z_{\pro\et} \to X_{\pro\et}$ as above.
	Let $H \in \SmGp' / X$.
	Then the natural reduction morphism $i^{\ast} H \to H \times_{X} Z$ induces an isomorphism
	$(i^{\ast} H)^{\wedge} \isomto (H \times_{X} Z)^{\wedge}$
	in $D(Z_{\pro\et})$.
\end{Prop}

\begin{proof}
	The reduction morphism $i^{\ast} H \to H \times_{X} Z$ is surjective by smoothness.
	We need to show that the multiplication by $l$
	on the kernel of $i^{\ast} H \to H \times_{X} Z$ in $\Ab(Z_{\pro\et})$
	is an isomorphism.
	We may assume that $X$ is strictly henselian and $U = \Spec K$.
	Since $Z$ is then a geometric point and $H$ locally of finite type,
	it is enough to show that
	$l \colon \Ker(H(X) \onto H(Z))$ is bijective.
	By dividing $H$ by the schematic closure of the identity section of $H \times_{X} U$,
	we may assume that $H$ is a separated scheme (\cite[Prop.\ 3.3.5]{Ray70a}).
	The multiplication by $l$ is an \'etale morphism on $H$ by \cite[7.3/2 (b)]{BLR90}.
	In particular, $\Ker(l) \subset H$ is a separated \'etale group scheme over $X$.
	Hence the map $\Ker(l)(X) \to \Ker(l)(Z)$ is bijective since $X$ is henselian.
	Therefore $l \colon \Ker(H(X) \onto H(Z))$ is injective.
	Let $a \in \Ker(H(X) \onto H(Z))$.
	Then the inverse image $l^{-1}(a) \subset H$ is a separated \'etale $X$-scheme
	whose special fiber contains $0$ (the identity element).
	Hence $l^{-1}(a)(X)$ is non-empty since $X$ is henselian.
	Thus $l \colon \Ker(H(X) \onto H(Z))$ is also surjective.
	This implies the result.
\end{proof}

\begin{Prop}
	The objects $\widehat{\mathcal{N}}(M)$, $\widehat{\mathcal{N}}_{0}(M)$ of $D(X_{\pro\et})$
	and the object $\widehat{\mathcal{P}}(M)$ of $D(Z_{\pro\et})$ are constructible.
	We have $\widehat{\mathcal{P}}(M) = \mathcal{P}(M) \tensor \Z_{l}$.
	We have a canonical distinguished triangle
		\[
				\widehat{\mathcal{N}}_{0}(M)
			\to
				\widehat{\mathcal{N}}(M)
			\to
				i_{\ast} \widehat{\mathcal{P}}(M)
		\]
	in $D(X_{\pro\et})$.
\end{Prop}

\begin{proof}
	The statements for $\widehat{\mathcal{N}}(M)$ and $\widehat{\mathcal{N}}_{0}(M)$
	follow from Prop.\ \ref{prop: ell adic realization is constructible}
	and \cite[Prop.\ 3.5.1 (2)]{BS15}.
	Since $\mathcal{P}(M) \in D^{b}(\EtGpf / Z)$,
	it follows that $\widehat{\mathcal{P}}(M) = \mathcal{P}(M) \tensor \Z_{l}$ is constructible.
	Applying the derived completion to the triangle in Def.\ \ref{def: Neron components},
	we get the stated distinguished triangle.
\end{proof}

\begin{Prop} \label{prop: pull of connected Neron and shriek pull of Neron}
	We have canonical isomorphisms
		\[
				i^{\ast} \widehat{\mathcal{N}}_{0}(M)
			\cong
				T_{l} \mathcal{G}_{Z 0}[1],
			\quad
				i^{!} \widehat{\mathcal{N}}(M)
			\cong
				T_{l} i^{\ast} (R^{1} j_{\ast} M) [-1],
		\]
	which are shifts of $\Z_{l}$-lattices over $Z$.
\end{Prop}

\begin{proof}
	By \cite[Rmk.\ 6.5.10]{BS15}, the derived completion commutes with $i^{\ast}$ and $j^{\ast}$.
	By the distinguished triangle $i^{!} \to i^{\ast} \to i^{\ast} R j_{\ast} j^{\ast}$
	of functors $D(X_{\pro\et}) \to D(Z_{\pro\et})$ (\cite[Lem.\ 6.1.16]{BS15}),
	we know that the derived completion also commutes with $i^{!}$.
	We have $i^{\ast} \mathcal{Y}_{0} = 0$.
	Also
		$
				i^{\ast} \mathcal{G}_{0}^{\wedge}
			=
				\mathcal{G}_{Z 0}^{\wedge}
			=
				T_{l} \mathcal{G}_{Z 0}[1]
		$
	by Prop.\ \ref{prop: sheaf pull and fiber after derived completion}
	and the fact that the smooth group scheme $\mathcal{G}_{Z 0}$ with connected fibers is $l$-divisible.
	Hence $i^{\ast} \widehat{\mathcal{N}}_{0}(M) = T_{l} \mathcal{G}_{0 Z}[1]$.
	We have a distinguished triangle
		\[
				i^{!} \mathcal{N}(M)
			\to
				i^{\ast} \mathcal{N}(M)
			\to
				i^{\ast} R j_{\ast} M.
		\]
	We have
		$
				i^{\ast} \mathcal{N}(M)
			=
				\tau_{\le 0}
				i^{\ast} R j_{\ast} M
		$.
	We claim that the $l$-primary part of $i^{\ast} R^{n} j_{\ast} M$ is divisible for $n = 1$ and zero for $n \ge 2$.
	This implies
		\[
				i^{!} \widehat{\mathcal{N}}(M)
			=
				(\tau_{\ge 1} i^{\ast} R j_{\ast} M)^{\wedge}[-1]
			=
				(i^{\ast} R^{1} j_{\ast} M)^{\wedge}[-2]
			=
				T_{l} i^{\ast} (R^{1} j_{\ast} M) [-1]
		\]
	as desired.
	
	Now we prove the above claim.
	We may assume that $X$ is strictly henselian and $U = \Spec K$.
	We need to show that the $l$-primary part of $H^{n}(K, M)$ is divisible for $n = 1$ and zero for $n \ge 2$.
	It is enough to show this for the case $M = G$ and the case $M = Y[1]$.
	Let $C$ be the $l$-primary part of $G_{\tor}$ or $Y \tensor \Q_{l} / \Z_{l}$.
	We need to show that $H^{n}(K, C)$ is divisible for $n = 1$ and zero for $n \ge 2$.
	Since $C$ is $l$-divisible with finite $l$-torsion part,
	this follows from the fact that the $l$-cohomological dimension of $K$ is $1$
	(\cite[II, \S 4.3, Prop.\ 12]{Ser02}).
\end{proof}

\begin{Prop} \label{prop: cohomology of ell adic realization of Neron models}
	The objects $i^{!} \widehat{\mathcal{N}}_{0}(M)$ and $(i^{\ast} \widehat{\mathcal{N}}(M))^{\vee}$ are
	concentrated in degrees $0$ and $1$.
	The $H^{0}$ of these objects are $\Z_{l}$-lattices.
	The objects $i^{\ast} \widehat{\mathcal{N}}_{0}(M)$ and $(i^{!} \widehat{\mathcal{N}}(M))^{\vee}$ are
	concentrated in degree $-1$.
\end{Prop}

\begin{proof}
	We have a distinguished triangle
		\[
				i^{!} \widehat{\mathcal{N}}_{0}(M)
			\to
				i^{!} \widehat{\mathcal{N}}(M)
			\to
				\widehat{\mathcal{P}}(M).
		\]
	Hence the statement about $i^{!} \widehat{\mathcal{N}}_{0}(M)$
	follows from Prop.\ \ref{prop: fibers of Nerom components}
	and \ref{prop: pull of connected Neron and shriek pull of Neron}.
	We also have a distinguished triangle
		\[
				i^{\ast} \widehat{\mathcal{N}}_{0}(M)
			\to
				i^{\ast} \widehat{\mathcal{N}}(M)
			\to
				\widehat{\mathcal{P}}(M).
		\]
	Hence by the same propositions, we know that
	$i^{\ast} \widehat{\mathcal{N}}(M)$ is concentrated in degrees $-1$, $0$,
	whose $H^{-1}$ is a $\Z_{l}$-lattice.
	This implies the statement about $(i^{\ast} \widehat{\mathcal{N}}(M))^{\vee}$.
	The rest is already in Prop.\ \ref{prop: pull of connected Neron and shriek pull of Neron}.
\end{proof}

\begin{Prop}
	For $F, G \in D(X_{\pro\et})$, we have
		\[
				R \sheafhom_{X_{\pro\et}}(F, G)^{\wedge}
			=
				R \sheafhom_{X_{\pro\et}}(\widehat{F}, \widehat{G}).
		\]
\end{Prop}

\begin{proof}
	The derived tensor product $(\var) \tensor^{L} \Z / l^{n} \Z$ is given by
	the mapping cone of multiplication by $l^{n}$,
	which commutes with $R \sheafhom_{X_{\pro\et}}(F, \var)$.
	The derived inverse limit $R \invlim$ also commutes with $R \sheafhom_{X_{\pro\et}}(F, \var)$.
	Hence we have
		\begin{align*}
					R \sheafhom_{X_{\pro\et}}(F, G)^{\wedge}
			&	=
					R \invlim_{n}
					R \sheafhom_{X_{\pro\et}}(F, G \tensor^{L} \Z / l^{n} \Z)
			\\
			&	=
					R \invlim_{n}
					R \sheafhom_{X_{\pro\et}}(\widehat{F}, G \tensor^{L} \Z / l^{n} \Z)
			\\
			&	=
					R \sheafhom_{X_{\pro\et}}(\widehat{F}, \widehat{G}),
		\end{align*}
	where the second equality comes from the fact that
	the mapping cone of $F \to \widehat{F}$ is uniquely $l$-divisible.
\end{proof}

Therefore the morphism $\zeta_{0 M}$ induces a morphism
	\[
			\widehat{\zeta}_{0 M}
		\colon
			\widehat{\mathcal{N}}_{0}(M^{\vee})
		\to
			R \sheafhom_{X_{\pro\et}}(\widehat{\mathcal{N}}(M), \Z_{l}(1)[2])
		=
			\widehat{\mathcal{N}}(M)^{\vee}
	\]
via derived completion.
Its pullback
	$
			(M^{\vee})^{\wedge}
		\to
			\widehat{M}^{\vee}
	$
to $U_{\pro\et}$ is an isomorphism
since $(M^{\vee})^{\wedge}[-1] = T_{l} M^{\vee}$ and $\widehat{M}[-1] = T_{l} M$ are dual to each other
after Tate twist.

\begin{Prop} \label{prop: duality for Neron holds ell adic integrally}
	The morphism $\widehat{\zeta}_{0 M}$ is an isomorphism.
	In $H^{-1} i^{\ast}$, it induces a perfect pairing
		\[
				T_{l} \mathcal{G}_{Z 0}'
			\longleftrightarrow
				T_{l} i^{\ast} R^{1} j_{\ast} M
		\]
	of $\Z_{l}$-lattices over $Z$.
\end{Prop}

\begin{proof}
	The statement about $H^{-1} i^{\ast}$ is a consequence of
	Prop.\ \ref{prop: pull of connected Neron and shriek pull of Neron},
	once we show that $\widehat{\zeta}_{0 M}$ is an isomorphism.
	We have
		\[
				i^{\ast} \bigl(
					\widehat{\mathcal{N}}(M)^{\vee}
				\bigr)
			=
				\bigl(
					i^{!} \widehat{\mathcal{N}}(M)
				\bigr)^{\vee}
		\]
	by Verdier duality.
	Since
		\[
				j^{\ast} \widehat{\zeta}_{0 M}
			\colon
				j^{\ast} \bigl(
					\widehat{\mathcal{N}}_{0}(M^{\vee}
				\bigr)
			\to
				j^{\ast} \bigl(
					\widehat{\mathcal{N}}(M)^{\vee}
				\bigr)
		\]
	is an isomorphism,
	it is enough to show that the morphism
		\[
				i^{\ast} \widehat{\zeta}_{0 M}
			\colon
				i^{\ast} \bigl(
					\widehat{\mathcal{N}}_{0}(M^{\vee})
				\bigr)
			\to
				i^{\ast} \bigl(
					\widehat{\mathcal{N}}(M)^{\vee}
				\bigr)
			=
				\bigl(
					i^{!} \widehat{\mathcal{N}}(M)
				\bigr)^{\vee}
		\]
	is an isomorphism.
	The morphism $\widehat{\zeta}_{0 M}$ induces
	a canonical morphism of distinguished triangles
		\[
			\begin{CD}
					i^{\ast} R j_{\ast} \bigl(
						\widehat{\mathcal{N}}_{0}(M^{\vee})
					\bigr)[-1]
				@>>>
					i^{!} \bigl(
						\widehat{\mathcal{N}}_{0}(M^{\vee})
					\bigr)
				@>>>
					i^{\ast} \bigl(
						\widehat{\mathcal{N}}_{0}(M^{\vee})
					\bigr)
				\\
				@VV \wr V
				@VVV
				@VVV
				\\
					\bigl(
						i^{\ast} R j_{\ast} \widehat{\mathcal{N}}(M)
					\bigr)^{\vee}
				@>>>
					\bigl(
						i^{\ast} \widehat{\mathcal{N}}(M)
					\bigr)^{\vee}
				@>>>
					\bigl(
						i^{!} \widehat{\mathcal{N}}(M)
					\bigr)^{\vee}.
			\end{CD}
		\]
	Denote the upper triangle by $E \to F \to G$ and lower by $E' \to F' \to G'$.
	Then by Prop.\ \ref{prop: cohomology of ell adic realization of Neron models},
	we have a commutative diagram
		\[
			\begin{CD}
					H^{1} E
				@>> \sim >
					H^{1} F
				\\
				@VV \wr V
				@VVV
				\\
					H^{1} E'
				@>\sim >>
					H^{1} F'
			\end{CD}
		\]
	and a commutative diagram with exact rows
		\[
			\begin{CD}
					0
				@>>>
					H^{-1} G
				@>>>
					H^{0} E
				@>>>
					H^{0} F
				@>>>
					0
				\\
				@. @VVV @VV \wr V @VVV
				\\
					0
				@>>>
					H^{-1} G'
				@>>>
					H^{0} E'
				@>>>
					H^{0} F'
				@>>>
					0,
			\end{CD}
		\]
	and the cohomologies of $E, F, G, E', F', G'$ are zero in all other degrees.
	The first diagram shows that the morphism $H^{1} F \to H^{1} F'$ is an isomorphism.
	The second diagram shows that $H^{0} F \to H^{0} F'$ is surjective.
	It is an isomorphism up to torsion by
	Prop.\ \ref{prop: duality holds rationally}.
	It is also a morphism between $\Z_{l}$-lattices
	by Prop.\ \ref{prop: cohomology of ell adic realization of Neron models}.
	These imply that $H^{0} F \to H^{0} F'$ is an isomorphism.
	Therefore $H^{-1} G \to H^{-1} G'$ is also an isomorphism.
	Hence $F \to F'$ and $G \to G'$ are both isomorphisms.
	This proves that $\widehat{\zeta}_{0 M}$ is an isomorphism.
\end{proof}

\begin{Prop} \label{prop: duality for components holds ell adic integrally}
	The morphism $\mathcal{P}(M^{\vee}) \to \mathcal{P}(M)^{\LDual}[1]$
	becomes an isomorphism after tensoring with $\Z_{l}$.
\end{Prop}

\begin{proof}
	The derived completion of the diagram in
	Prop.\ \ref{prop: morphism between connected etale triangle to dual such}
	gives a morphism of distinguished triangles
		\[
			\begin{CD}
					\widehat{\mathcal{N}}_{0}(M^{\vee})
				@>>>
					\widehat{\mathcal{N}}(M^{\vee})
				@>>>
					i_{\ast} \widehat{\mathcal{P}}(M^{\vee})
				\\
				@VV \widehat{\zeta}_{0 M} V
				@VV \widehat{\zeta}_{M} V
				@VV \widehat{\eta}_{M} V
				\\
					(\widehat{\mathcal{N}}(M))^{\vee}
				@>>>
					(\widehat{\mathcal{N}}_{0}(M))^{\vee}
				@>>>
					i_{\ast} \widehat{\mathcal{P}}(M)^{\vee}[1].
			\end{CD}
		\]
	The morphism $\widehat{\zeta}_{0 M}$ is an isomorphism
	by Prop.\ \ref{prop: duality for Neron holds ell adic integrally}.
	The morphism $\widehat{\zeta}_{M}$ can be obtained by applying the dual $(\var)^{\vee}$ to $\widehat{\zeta}_{0 M^{\vee}}$
	by the definition of $\zeta_{0 M}$.
	Hence $\widehat{\zeta}_{M}$ is an isomorphism.
	(This is not a consequence of the fact that
	$\zeta_{M} \colon \mathcal{N}(M^{\vee}) \to \mathcal{N}_{0}(M)^{\vee}$
	is an isomorphism,
	since derived completion does not commutes with the truncation $\tau_{\le 0}$
	that appears in Def.\ \ref{def: dual of Neron model}.)
	Therefore $\widehat{\eta}_{M}$ is an isomorphism.
\end{proof}

The following finishes the proof of Thm.\ \ref{main: duality} \eqref{main: item: after inverting res char}.

\begin{Prop}
	Any of the kernel or cokernel of $H^{n} \zeta_{M}$ and $H^{n} \eta_{M}$ for any $n$
	is of the form $i_{\ast} N$ for some finite \'etale group scheme $N$ over $Z$
	whose fiber over any $x \in Z$ has order a power of the residual characteristic exponent at $x$.
\end{Prop}

\begin{proof}
	This follows from Prop.\ \ref{prop: duality for components holds ell adic integrally}
	and \ref{prop: duality holds up to finite etale skyscrapers}.
\end{proof}


\section{Duality for cohomology of N\'eron models and perfectness for $p$-part}
\label{sec: Duality for cohomology of Neron models and perfectness for p-part}

We continue working in Situation \ref{sit: to consider Neron models, notation for M and its dual}.
Assume that $X = \Spec \Order_{K}$ is the spectrum of a complete discrete valuation ring $\Order_{K}$
with \emph{perfect} residue field $k$ of characteristic $p > 0$ and $U = \Spec K$ its generic point.
Below we use the same notation as \cite[\S 2.1]{Suz14}, \cite[\S 2.3]{Suz18}
about the ind-rational pro-\'etale site of $k$.
We also write $Z = x$, $\Order_{K} = \widehat{\Order}_{x}$, $K = \widehat{K}_{x}$
in order to match the notation in \cite[\S 2.5]{Suz18}.
Let $k^{\ind\rat}$ be the category of ind-rational $k$-algebras
and $\Spec k^{\ind\rat}_{\pro\et}$ the ind-rational pro-\'etale site of $k$
defined in \cite[\S 2.1]{Suz14}, \cite[\S 2.3]{Suz18}.
Let
	\begin{gather*}
				R \alg{\Gamma}(\Order_{K}, \var),
				R \alg{\Gamma}_{x}(\Order_{K}, \var)
			\colon
				D(\Order_{K, \fppf})
			\to
				D(k^{\ind\rat}_{\pro\et})
		\\
				R \alg{\Gamma}(K, \var),
			\colon
				D(K_{\fppf})
			\to
				D(k^{\ind\rat}_{\pro\et})
	\end{gather*}
be the functors defined in \cite[\S 2.5]{Suz18}.
The composites of them with the $n$-th cohomology object functor $H^{n}$ for any $n$ is denoted by
$\alg{H}^{n}(\Order_{K}, \var)$, $\alg{H}^{n}_{x}(\Order_{K}, \var)$
and $\alg{H}^{n}(K, \var)$, respectively,
and we set $\alg{\Gamma}(\Order_{K}, \var) = \alg{H}^{0}(\Order_{K}, \var)$,
$\alg{\Gamma}(K, \var) = \alg{H}^{0}(K, \var)$.
We denote the Serre dual functor by
$(\var)^{\SDual} = R \sheafhom_{k^{\ind\rat}_{\pro\et}}(\var, \Z)$ (\cite[\S 2.4]{Suz14}).
An object $C \in D(k^{\ind\rat}_{\pro\et})$ is said to be Serre reflexive
if the canonical morphism $C \to C^{\SDual \SDual}$ is an isomorphism.

We have the canonical trace morphism
	\[
			R \alg{\Gamma}(K, \Gm)
		\to
			R \alg{\Gamma}_{x}(\Order_{K}, \Gm)[1]
		=
			\Z
	\]
by \cite[Eq.\ (2.5.7)]{Suz18}.
The morphism
	\[
			M^{\vee}
		\to
			R \sheafhom_{K_{\fppf}}(M, \Gm[1])
	\]
induces morphisms
	\begin{align*}
				R \alg{\Gamma}(K, M^{\vee})
		&	\to
				R \sheafhom_{k^{\ind\rat}_{\pro\et}} \bigl(
					R \alg{\Gamma}(K, M),
					R \alg{\Gamma}(K, \Gm[1])
				\bigr)
		\\
		&	\to
				R \sheafhom_{k^{\ind\rat}_{\pro\et}} \bigl(
					R \alg{\Gamma}(K, M),
					\Z[1]
				\bigr)
			=
				R \alg{\Gamma}(K, M)^{\SDual}[1]
	\end{align*}
in $D(k^{\ind\rat}_{\pro\et})$ as in \cite[\S 2.5]{Suz18}.
Its Serre dual (when $M$ and $M^{\vee}$ are switched)
	\[
			R \alg{\Gamma}(K, M^{\vee})^{\SDual \SDual}
		\to
			R \alg{\Gamma}(K, M)^{\SDual}[1]
	\]
is an isomorphism by \cite[Thm.\ (9.1)]{Suz14} and the comparison results in \cite[Appendix A]{Suz18}.
This result is the main input for the results of this section.

By Prop.\ \ref{prop: fppf and smooth sites},
the distinguished triangle in Def.\ \ref{def: Neron components} and
the commutative diagram in Prop.\ \ref{prop: diagram of duality morphisms for Neron}
can be translated in the fppf site.

\begin{Prop} \label{prop: component cohom and localization triangles}
	The distinguished triangle in Def.\ \ref{def: Neron components}
	and the localization triangle in \cite[\S 2.5]{Suz18}
	(i.e.\ the definition of $R \alg{\Gamma}_{x}$ as a mapping cone)
	 induce a commutative diagram of distinguished triangles
		\[
			\begin{CD}
					R \alg{\Gamma}_{x}(\Order_{K}, \mathcal{N}_{0}(M))
				@>>>
					R \alg{\Gamma}_{x}(\Order_{K}, \mathcal{N}(M))
				@>>>
					\mathcal{P}(M)
				\\
				@VVV
				@VVV
				@|
				\\
					R \alg{\Gamma}(\Order_{K}, \mathcal{N}_{0}(M))
				@>>>
					R \alg{\Gamma}(\Order_{K}, \mathcal{N}(M))
				@>>>
					\mathcal{P}(M)
				\\
				@VVV
				@VVV
				@VVV
				\\
					R \alg{\Gamma}(K, M)
				@=
					R \alg{\Gamma}(K, M)
				@>>>
					0.
			\end{CD}
		\]
\end{Prop}

\begin{proof}
	Obvious.
\end{proof}

\begin{Prop} \BetweenThmAndList \label{prop: description of cohom of Neron models}
	\begin{enumerate}
		\item \label{item: cohom of N zero without support}
			About $R \alg{\Gamma}(\Order_{K}, \mathcal{N}_{0}(M))^{\SDual \SDual}$:
			The $H^{-1}$ is the Tate module $T \alg{\Gamma}(K, G)_{\sAb}$
			of the maximal semi-abelian subgroup $\alg{\Gamma}(K, G)_{\sAb}$
			of the proalgebraic group $\alg{\Gamma}(K, G)$.
			The $H^{0}$ is the quotient $\alg{\Gamma}(K, G)_{0 / \sAb}$
			of the identity component $\alg{\Gamma}(K, G)_{0}$ by $\alg{\Gamma}(K, G)_{\sAb}$.
		\item \label{item: cohom of N with support}
			About $R \alg{\Gamma}_{x}(\Order_{K}, \mathcal{N}(M))^{\SDual}$:
			The $H^{-1}$ is the Pontryagin dual $\pi_{0}(\alg{H}^{1}(K, M))^{\PDual}$
			of the component group $\pi_{0}(\alg{H}^{1}(K, M))$
			of the ind-algebraic group $\alg{H}^{1}(K, M)$.
			The $H^{0}$ is the dual
				\[
						\alg{H}^{1}(K, M)_{0}^{\SDual'}
					:=
						\sheafext_{k^{\ind\rat}_{\pro\et}}^{1} \bigl(
							\alg{H}^{1}(K, M)_{0},
							\Q / \Z
						\bigr)
				\]
			of the identity component $\alg{H}^{1}(K, M)_{0}$.
		\item \label{item: cohom of N zero with support}
			About $R \alg{\Gamma}_{x}(\Order_{K}, \mathcal{N}_{0}(M))$:
			The $H^{0}$ is $H^{-1} \mathcal{P}(M)$.
			The $H^{1}$ is $H^{0} \mathcal{P}(M)$.
			The $H^{2}$ is $\alg{H}^{1}(K, M)$.
		\item \label{item: cohom of N without support}
			About $R \alg{\Gamma}(\Order_{K}, \mathcal{N}(M))^{\SDual}$:
			The $H^{0}$ is $H^{0}(\mathcal{P}(M)^{\LDual})$.
			The $H^{1}$ is $H^{1}(\mathcal{P}(M)^{\LDual})$.
			The $H^{2}$ is the dual
				\[
						\alg{\Gamma}(K, G)_{0}^{\SDual'}
					:=
						\sheafext_{k^{\ind\rat}_{\pro\et}}^{1} \bigl(
							\alg{\Gamma}(K, G)_{0},
							\Q / \Z
						\bigr)
				\]
			of the identity component $\alg{\Gamma}(K, G)_{0}$.
	\end{enumerate}
	In both cases \eqref{item: cohom of N zero without support} and \eqref{item: cohom of N with support},
	the $H^{-1}$ is a pro-finite-\'etale group scheme over $k$
	and the $H^{0}$ is a connected unipotent proalgebraic group over $k$.
	In both cases \eqref{item: cohom of N zero with support} and \eqref{item: cohom of N without support},
	the $H^{0}$ is a lattice, $H^{1} \in \EtGpf / k$
	and $H^{2}$ an ind-algebraic group with unipotent identity component.
	All the four complexes above are Serre reflexive.
	We have $H^{n} = 0$ for all the complexes for all other degrees.
\end{Prop}

\begin{proof}
	\eqref{item: cohom of N zero without support}
	Since $\mathcal{Y}_{0}$ is the extension by zero of the \'etale group $Y$,
	we have $R \alg{\Gamma}(\Order_{K}, \mathcal{Y}_{0}) = 0$
	by \cite[Prop.\ (5.2.3.4)]{Suz14}.
	We have
		\[
				R \alg{\Gamma}(\Order_{K}, \mathcal{G}_{0})
			=
				\alg{\Gamma}(\Order_{K}, \mathcal{G}_{0})
			=
				\alg{\Gamma}(\Order_{K}, \mathcal{G})_{0}
			=
				\alg{\Gamma}(K, G)_{0},
		\]
	where the first equality is \cite[Prop.\ (3.4.1)]{Suz14} (with the smoothness of $\mathcal{G}_{0}$),
	the second \cite[Prop.\ (3.4.2) (a)]{Suz14}
	and the third \cite[Prop.\ (3.1.3) (c)]{Suz14}.
	Hence $R \alg{\Gamma}(\Order_{K}, \mathcal{N}_{0}(M)) = \alg{\Gamma}(K, G)_{0}$,
	which is a connected proalgebraic group by \cite[Prop.\ (3.4.2) (a)]{Suz14}.
	Therefore the description of its double dual follows from
	\cite[Prop.\ (2.4.1) (d) and Footnote 7]{Suz14}.
	
	\eqref{item: cohom of N with support}
	We have
		\[
				R \alg{\Gamma}(\Order_{K}, \mathcal{N}(M))
			=
				\tau_{\le 0}
				R \alg{\Gamma}(\Order_{K}, R j_{\ast} M)
			=
				\tau_{\le 0}
				R \alg{\Gamma}(K, M)
		\]
	by \cite[Prop.\ (3.4.1)]{Suz14} (truncation commutes with exact functors).
	Hence one of the distinguished triangles in Prop.\ \ref{prop: component cohom and localization triangles}
	shows that
		\[
				R \alg{\Gamma}_{x}(\Order_{K}, \mathcal{N}(M))
			=
				\alg{H}^{1}(K, M)[-2]
		\]
	since $R \alg{\Gamma}(K, M)$ is concentrated in degrees $-1, 0, 1$
	(see \cite[first paragraph of \S 9]{Suz14}).
	By loc.\ cit., we know that $\alg{H}^{1}(K, M)$ is an ind-algebraic group with unipotent identity component.
	Hence \cite[Prop.\ (2.4.1) (b)]{Suz14} gives the required description of its Serre dual.
	
	\eqref{item: cohom of N zero with support}
	One of the distinguished triangles in Prop.\ \ref{prop: component cohom and localization triangles}
	and what we saw right above give a distinguished triangle
		\[
				R \alg{\Gamma}_{x}(\Order_{K}, \mathcal{N}_{0}(M))
			\to
				\alg{H}^{1}(K, M)[-2]
			\to
				\mathcal{P}(M).
		\]
	The result follows from this.
	
	\eqref{item: cohom of N without support}
	One of the distinguished triangles in Prop.\ \ref{prop: component cohom and localization triangles}
	and what we saw in the proof of \eqref{item: cohom of N zero without support} above
	give a distinguished triangle
		\[
				\mathcal{P}(M)^{\LDual}
			\to
				R \alg{\Gamma}(\Order_{K}, \mathcal{N}(M))^{\SDual}
			\to
				\alg{\Gamma}(K, G)_{0}^{\SDual}.
		\]
	We have
		\begin{align*}
					\alg{\Gamma}(K, G)_{0}^{\SDual}
			&	=
					R \sheafhom_{k^{\ind\rat}_{\pro\et}} \bigl(
						\alg{\Gamma}(K, G)_{0},
						\Q / \Z
					\bigr)[-1]
			\\
			&	=
					\sheafext_{k^{\ind\rat}_{\pro\et}}^{1} \bigl(
						\alg{\Gamma}(K, G)_{0},
						\Q / \Z
					\bigr)[-2]
			\\
			&	=
					\alg{\Gamma}(K, G)_{0}^{\SDual'}[-2]
		\end{align*}
	by \cite[Prop.\ (2.3.3) (d), (2.4.1) (a)]{Suz14}.
	Hence the statements in \eqref{item: cohom of N without support} follow.
	The group $\alg{\Gamma}(K, G)_{0}^{\SDual'}$ is an ind-algebraic group with unipotent identity component
	by \cite[(2.4.1) (d)]{Suz14}.
	
	We can check that the cohomology objects of all the four complexes are Serre reflexive
	using \cite[(2.4.1) (b)]{Suz14}.
	Hence the four complexes themselves are Serre reflexive.
\end{proof}

By Prop.\ \ref{prop: fppf and smooth sites},
the morphism
	\[
			\mathcal{N}(M^{\vee})
		\to
			R \sheafhom_{\Order_{K, \sm}}(\mathcal{N}_{0}(M), \Gm[1])
	\]
in $D(\Order_{K, \sm})$ given in Def.\ \ref{def: duality morphisms for Neron}
induces a morphism
	\[
			\mathcal{N}(M^{\vee})
		\to
			R \sheafhom_{\Order_{K, \fppf}}(\mathcal{N}_{0}(M), \Gm[1])
	\]
in $D(\Order_{K, \fppf})$.
Hence \cite[Prop.\ (3.3.8)]{Suz14} and the trace morphism
give a morphism of distinguished triangles
	\[
		\begin{CD}
				R \alg{\Gamma}(\Order_{K}, \mathcal{N}(M))
			@>>>
				R \alg{\Gamma}(K, M)
			@>>>
				R \alg{\Gamma}_{x}(\Order_{K}, \mathcal{N}(M))[1]
			\\
			@VVV
			@VVV
			@VVV
			\\
				R \alg{\Gamma}_{x}(\Order_{K}, \mathcal{N}_{0}(M^{\vee}))^{\SDual}
			@>>>
				R \alg{\Gamma}(K, M^{\vee})^{\SDual}[1]
			@>>>
				R \alg{\Gamma}(\Order_{K}, \mathcal{N}_{0}(M^{\vee}))^{\SDual}[1]
		\end{CD}
	\]
Applying $\SDual$, shifting by one
and using the Serre reflexibility of $R \alg{\Gamma}_{x}(\Order_{K}, \mathcal{N}_{0}(M^{\vee}))$
(Prop.\ \ref{prop: description of cohom of Neron models}),
we have a morphism of distinguished triangles
	\begin{equation} \label{eq: duality morphisms for cohomology of Neron}
		\begin{CD}
				R \alg{\Gamma}(\Order_{K}, \mathcal{N}_{0}(M^{\vee}))^{\SDual \SDual}
			@>>>
				R \alg{\Gamma}(K, M^{\vee})^{\SDual \SDual}
			@>>>
				R \alg{\Gamma}_{x}(\Order_{K}, \mathcal{N}_{0}(M^{\vee}))[1]
			\\
			@VVV
			@VVV
			@VVV
			\\
				R \alg{\Gamma}_{x}(\Order_{K}, \mathcal{N}(M))^{\SDual}
			@>>>
				R \alg{\Gamma}(K, M)^{\SDual}[1]
			@>>>
				R \alg{\Gamma}(\Order_{K}, \mathcal{N}(M))^{\SDual}[1].
		\end{CD}
	\end{equation}
To simplify the notation,
we denote the upper triangle by $C \to D \to E$ and lower by $C' \to D' \to E'$.
As noted earlier,
the middle vertical morphism is an isomorphism by \cite[Thm.\ (9.1)]{Suz14},
so $D \isomto D'$.
The above diagram induces a morphism from the long exact sequence of cohomologies of $C \to D \to E$
to the long exact sequence of cohomologies of $C' \to D' \to E'$.
We can spell it out using Prop.\ \ref{prop: description of cohom of Neron models} as follows:
	\begin{gather*}
			\begin{CD}
					0
				@>>>
					T \alg{\Gamma}(K, G')_{\sAb}
				@>>>
					H^{-1} D
				@>>>
					H^{-1} \mathcal{P}(M^{\vee})
				\\
				@.
				@VVV
				@VV \wr V
				@VVV
				\\
					0
				@>>>
					(\pi_{0} \alg{H}^{1}(K, M))^{\PDual}
				@>>>
					H^{-1} D'
				@>>>
					H^{0}(\mathcal{P}(M)^{\LDual})
			\end{CD}
		\\
			\begin{CD}
				@>>>
					\alg{\Gamma}(K, G')_{0 / \sAb}
				@>>>
					H^{0} D
				@>>>
					H^{0} \mathcal{P}(M)
				@>>>
					0
				\\
				@.
				@VVV
				@VV \wr V
				@VVV
				@.
				\\
				@>>>
					\alg{H}^{1}(K, M)_{0}^{\SDual'}
				@>>>
					H^{0} D'
				@>>>
					H^{1}(\mathcal{P}(M)^{\LDual})
				@>>>
					0,
			\end{CD}
		\\
			\begin{CD}
					0
				@>>>
					H^{1} D
				@>>>
					\alg{H}^{1}(K, M^{\vee})
				@>>>
					0
				\\
				@.
				@VV \wr V
				@VVV
				@.
				\\
					0
				@>>>
					H^{1} D'
				@>>>
					\alg{\Gamma}(K, G)_{0}^{\SDual'}
				@>>>
					0.
			\end{CD}
	\end{gather*}
The upper, middle and lower diagrams are for the $H^{-1}$, $H^{0}$ and $H^{1}$, respectively.

\begin{Prop} \label{prop: duality for cohomology of Neron models}
	The morphism of distinguished triangles \eqref{eq: duality morphisms for cohomology of Neron}
	is an isomorphism of distinguished triangles.
	The $H^{1}$ of the right vertical isomorphism gives an isomorphism
		\begin{equation} \label{eq: Shafarevich isomorphism for Neron}
				\alg{H}^{1}(K, M^{\vee})
			\isomto
				\alg{\Gamma}(K, G)_{0}^{\SDual'}.
		\end{equation}
\end{Prop}

\begin{proof}
	The latter statement about $H^{1}$ is clear by the paragraph before the proposition.
	Applying $R \alg{\Gamma}_{x}(\Order_{K}, \var)$
	to the diagram in Prop.\ \ref{prop: morphism between connected etale triangle to dual such},
	and using \cite[Prop.\ 3.3.8]{Suz14} and the trace morphism,
	we have a morphism of distinguished triangles
		\begin{equation} \label{eq: morphism bw cohomology component triangles}
			\begin{CD}
					R \alg{\Gamma}_{x}(\Order_{K}, \mathcal{N}_{0}(M^{\vee}))
				@>>>
					R \alg{\Gamma}_{x}(\Order_{K}, \mathcal{N}(M^{\vee}))
				@>>>
					\mathcal{P}(M^{\vee})
				\\
				@VVV
				@VVV
				@VVV
				\\
					R \alg{\Gamma}(\Order_{K}, \mathcal{N}(M))^{\SDual}
				@>>>
					R \alg{\Gamma}(\Order_{K}, \mathcal{N}_{0}(M))^{\SDual}
				@>>>
					\mathcal{P}(M)^{\LDual}[1].
			\end{CD}
		\end{equation}
	The morphism in $H^{2}$ of the left vertical morphism is
	the isomorphism \eqref{eq: Shafarevich isomorphism for Neron}.
	The objects $\mathcal{P}(M)$ and $\mathcal{P}(M^{\vee})^{\LDual}[1]$ are concentrated in degrees $-1, 0$
	by Prop.\ \ref{prop: fibers of Nerom components}
	and the proof of Prop.\ \ref{prop: diagram of duality morphisms for Neron}.
	Hence the left horizontal two morphisms are both isomorphisms in $H^{2}$.
	The upper middle term $R \alg{\Gamma}_{x}(\Order_{K}, \mathcal{N}(M^{\vee}))$
	is concentrated in degree $2$ as we saw in the proof of
	Prop.\ \ref{prop: description of cohom of Neron models}
	\eqref{item: cohom of N with support}.
	As we saw in the proof of Prop.\ \ref{prop: description of cohom of Neron models}
	\eqref{item: cohom of N zero without support},
	the object $R \alg{\Gamma}(\Order_{K}, \mathcal{N}_{0}(M)) = \alg{\Gamma}(K, G)_{0}$
	is a connected proalgebraic group.
	Hence its Serre dual is concentrated in degree $2$
	by \cite[Prop.\ (2.4.1) (b)]{Suz14}.
	Therefore the lower middle term $R \alg{\Gamma}(\Order_{K}, \mathcal{N}_{0}(M))^{\SDual}$
	in the above diagram is also concentrated in degree $2$.
	Combining all these, we know that the middle vertical morphism in the above diagram is an isomorphism.
	Its Serre dual is $C \to C'$ with $M$ replaced by $M^{\vee}$.
	Therefore $C \to C'$ is an isomorphism.
	Hence $E \to E'$ is an isomorphism.
\end{proof}

\begin{Prop} \label{prop: duality for Neron components in local perfect residue field case}
	The morphism $\eta_{M}$ is an isomorphism.
\end{Prop}

\begin{proof}
	The right vertical morphism in \eqref{eq: morphism bw cohomology component triangles}
	is $\eta_{M}$ by Prop.\ \ref{prop: morphism between connected etale triangle to dual such}
	and the construction of \eqref{eq: morphism bw cohomology component triangles}.
	The middle vertical morphism is an isomorphism
	as seen in the proof of Prop.\ \ref{prop: duality for cohomology of Neron models}.
	The left vertical morphism is $E \to E'$ up to shift,
	which is an isomorphism by Prop.\ \ref{prop: duality for cohomology of Neron models}.
	Therefore $\eta_{M}$ is an isomorphism.
\end{proof}

The following finishes the proof of
Thm.\ \ref{main: duality} \eqref{main: item: perfect residue field case}
and hence of Thm.\ \ref{main: duality} itself.

\begin{Prop}
	Let $X$ be an irreducible Dedekind scheme
	and $j \colon U \into X$ a dense open subscheme with complement $Z$.
	Assume the residue fields of $Z$ are perfect.
	Then for any $M \in \mathcal{M}_{U}$,
	the morphism $\eta_{M}$ is an isomorphism.
\end{Prop}

\begin{proof}
	This follows from Prop.\ \ref{prop: reduction to str hensel case},
	\ref{prop: duality for components holds ell adic integrally}
	(for zero residual characteristics)
	and \ref{prop: duality for Neron components in local perfect residue field case}
	(for positive residual characteristics).
\end{proof}

\begin{Rmk}
	The right-hand side of \eqref{eq: Shafarevich isomorphism for Neron} depends only on $G$
	and not on $Y$.
	Hence the left-hand side actually depends only on $[Y' \to A']$ and not on $T'$.
	This can also be checked directly by noting that
	$\alg{H}^{n}(K, T') = 0$ for $n \ge 1$ by \cite[Prop.\ (3.4.3) (e)]{Suz14},
	the distinguished triangle
	$T' \to M^{\vee} \to [Y' \to A']$ and hence an isomorphism
	$\alg{H}^{1}(K, M^{\vee}) \isomto \alg{H}^{1}(K, [Y' \to A'])$.
	In particular, \eqref{eq: Shafarevich isomorphism for Neron} can be written as
		\[
				\alg{H}^{1}(K, [Y' \to A'])
			\isomto
				\alg{\Gamma}(K, G)_{0}^{\SDual'}.
		\]
	A similar remark exists for Prop.\ \ref{prop: duality for Neron holds ell adic integrally}.
\end{Rmk}


\end{document}